\documentclass[12pt]{article}
\usepackage{amsmath, amsfonts, amsthm, amssymb, color, hyperref, extarrows, enumitem, nicefrac,blindtext, mathtools,}

\newtheorem{theorem}{Theorem}[section]
\newtheorem{definition}[theorem]{Definition}
\newtheorem{cor}[theorem]{Corollary}
\newtheorem{lem}[theorem]{Lemma}
\newtheorem{pro}[theorem]{Proposition}

\numberwithin{equation}{section}

\newtheorem{remark}[theorem]{Remark}

\newtheorem{example}{Example}

\usepackage[hyperpageref]{backref}

\usepackage{cite}

\hypersetup{
	colorlinks   = true,
	citecolor    = magenta}
	
\begin{document}
\title{\vspace{-1cm} \bf Optimal Sobolev regularity of $\bar\partial$ on the Hartogs triangle  \rm}
\author{Yifei Pan\ \  and\ \  Yuan Zhang}
\date{}

\maketitle

\begin{abstract}
  In this paper, we show that for each $k\in \mathbb Z^+, p>4$, there exists  a solution operator $\mathcal T_k$ to the $\bar\partial$ problem on the Hartogs triangle that maintains the  same $W^{k, p}$ regularity as that of the data.  
  According to a Kerzman-type example, this   operator provides solutions with the optimal Sobolev regularity.

\end{abstract}

\renewcommand{\thefootnote}{\fnsymbol{footnote}}
\footnotetext{\hspace*{-7mm}
\begin{tabular}{@{}r@{}p{16.5cm}@{}}
& 2010 Mathematics Subject Classification.  Primary 32W05; Secondary 32A25.\\
& Key words and phrases. Hartogs triangle,  $\bar\partial$, Sobolev regularity, refined Muckenhoupt's class.   
\end{tabular}}

\section{Introduction}

 The Hartogs triangle   $$ \mathbb{H} = \{(z_1, z_2)\in \mathbb{C}^2: |z_1|< |z_2| < 1\}  $$
 is  a pseudoconvex domain with non-Lipschitz boundary. It serves as a model counterexample for many questions in several complex variables. For instance, it does not admit a Stein neighborhood basis  or a bounded plurisubharmonic exhaustion function. 
 Meanwhile,  Chaumat and Chollet showed in \cite{CC} that the corresponding $\bar\partial$ problem on $\mathbb{H}$ is not globally regular  in the sense that there is a smooth $\bar\partial$-closed $(0, 1)$-form $\mathbf f$   on $\overline{\mathbb{H} }$, such that  $\bar\partial u =\mathbf f$  has no smooth solution on $\overline{\mathbb{H} }$. Interestingly, at each H\"older level the $\bar\partial$ equation does  admit H\"older solutions with the same H\"older regularity as that of the data. For more properties on $\mathbb H$ please refer to a survey  \cite{Sh} of Shaw. On the other hand, the study of Sobolev regularity was initiated by Chakrabarti and Shaw in \cite{CS}, where they  carried out a weighted $L^2$-Sobolev estimate  for the canonical solution on $\mathbb H$. See also a recent work \cite{YZ} of  Yuan and the second author on weighted $L^p$-Sobolev estimates of $\bar\partial$ on general quotient domains.

 The goal of this paper is  to study the   optimal $\bar\partial $  regularity   on $\mathbb H$ at each  (unweighted)  Sobolev level. Recently, the optimal $L^p$ regularity of $\bar\partial$ on $\mathbb H$ was obtained by the second author in \cite{Zhang2}.  The following is our main theorem concerning the $W^{k, p}$ regularity, $ k\ge 1$. As demonstrated by  a Kerzman-type Example \ref{ex} (in Section 4), it gives  the optimal $W^{k,p}$ regularity in the sense that  for any $\epsilon>0$, there exists a $W^{k, p}$ datum which has no $W^{k, p+\epsilon}$ solution to $\bar\partial$ on $\mathbb H$.
 
\begin{theorem}\label{main}
For each $k\in \mathbb Z^+, 4<p<\infty$, there exists a solution operator $\mathcal T_k$ such that for any $\bar\partial$-closed $(0, 1)$ form $\mathbf f\in W^{k,p}(\mathbb H)$,  $\mathcal T_k\mathbf f\in W^{k,p}(\mathbb H)$ and solves  $\bar\partial u =\mathbf f$ on $\mathbb H$. Moreover, there exists a constant $C$ dependent only on   $k$ and $p$  such that   \begin{equation*}
    \|\mathcal T_k\mathbf f\|_{W^{k,p}(\mathbb H)}\le C\|\mathbf f\|_{W^{k,p}(\mathbb H)}.
\end{equation*} 
\end{theorem}
\medskip

The general idea of the proof is as follows. According to a heuristic procedure to treat the $\bar\partial$ problem on the Hartogs triangle $\mathbb H$,  one first  uses the biholomorphism between  the punctured bidisc and  $\mathbb H$ to pull back the data and  solve  $\bar\partial$ on the punctured bidisc, and then pushes the solutions forward   onto the Hartogs triangle.
As a consequence of this,  the corresponding Sobolev regularity of the  $\bar\partial$ problem  requires a weighted Sobolev regularity on product domains due to the presence of the nontrivial Jacobian of the biholomorphism. 
Based upon our recent weighted Sobolev result \cite{PZ2} about  Cauchy-type integrals, we first obtain the following Sobolev regularity for   $\bar\partial$ on product domains with respect to weights in some refined Muckenhoupt space $A_p^*$ (see Definition \ref{aps}).

\begin{theorem}\label{mainp}
Let $\Omega = D_1\times \cdots D_n, n\ge 2$, where each $D_j$ is a bounded domain in $\mathbb C$ with $C^{k,  1}$ boundary. There exists a solution operator $T$ such that for any $\bar\partial$-closed $(0, q)$ form $\mathbf f\in W^{k+n-2, p}(\Omega, \mu), k\in \mathbb Z^+, 1<p<\infty, \mu\in A_p^* $,  $T\mathbf f\in W^{k, p}(\Omega, \mu)$ and solves $\bar\partial u = \mathbf f$ on $\Omega$. Moreover, there exists a constant $C$ dependent only on $\Omega$, $k, p$ and the $A_p^*$ constant of  $\mu$  such that 
\begin{equation*} 
     \|T\mathbf f\|_{W^{k, p}(\Omega, \mu) } \le C\|\mathbf f\|_{W^{k+n-2, p}(\Omega, \mu)  }.
\end{equation*}
\end{theorem}
\medskip

As shown by Example \ref{ex2} (in Section 3),  Theorem \ref{mainp} gives the  optimal Sobolev regularity of solutions on product domains with dimension $n=2$.  Jin and Yuan   obtained in \cite{JY} a similar Sobolev estimate  for polydiscs in the case when  $\mu \equiv 1$ and $q=1$.   It is also worth pointing out that  the operator $T$ considered in   Theorem \ref{mainp} fails to maintain the $L^p$ (where $k=0$) regularity in general. See \cite{CM} of Chen and McNeal for a $\bar\partial$-closed (0,1) form $\mathbf f$ in $L^p(\triangle^2)$ such that that  $T\mathbf f$ fails to lie in $L^p(\triangle^2), p<2$. Instead, \cite{Zhang2} made use of the canonical solution  operator to provide an optimal weighted $L^p$ regularity for $\bar\partial$ on product domains in $\mathbb C^n$.

Theorem \ref{mainp} readily gives a   semi-weighted $L^p$-Sobolev estimate  below for a (fixed) solution operator  to $\bar\partial$ on $\mathbb H, p>2$.

\begin{cor}\label{main4}
There exists a solution operator $\mathcal T$ such that for any $\bar\partial$-closed $(0, 1)$ form $\mathbf f\in W^{k,p}(\mathbb H), k\in \mathbb Z^+, 2< p<\infty$,  $\mathcal T\mathbf f\in W^{k,p}(\mathbb H, |z_2|^{kp})$ and solves  $\bar\partial u =\mathbf f$ on $\mathbb H$. Moreover, there exists a constant $C$ dependent only on   $k$ and $p$  such that    \begin{equation*}
    \|\mathcal T\mathbf f\|_{W^{k,p}(\mathbb H, |z_2|^{kp})}\le C\|\mathbf f\|_{W^{k,p}(\mathbb H)}.
\end{equation*} 
\end{cor}
\medskip

The  estimate  in Corollary \ref{main4}  maintains  the Sobolev index $(k, p)$, and in particular improves  a   result   in \cite{YZ}.  We note that the $p>2$ assumption in the corollary  is due to the fact that the weight after pulling the data on $\mathbb H$ back  to the bidisc 
lies in $A^*_p$ only when  $p>2$, where Theorem \ref{mainp} can be applied.  Unfortunately, the solution operator $\mathcal  T$ here subjects to  some quantified loss in the exponent of the weight at each Sobolev level. Although this weight loss is not unexpected due to the global irregularity of $\bar\partial$ on $\mathbb H$, $ \mathcal T$ does not provide an optimal Sobolev regularity.


In order to obtain the optimal Sobolev regularity for $\bar\partial$ on $\mathbb H$, one needs to  further adjust the solution operator $ \mathcal T$ in Corollary \ref{main4} accordingly at different Sobolev levels. 
In fact, we apply  to $ \mathcal T$ a surgical procedure -- truncation by Taylor polynormials: one   on the data, and another  on the $\bar\partial$ solution on the punctured bidisc. The idea was initially introduced by Ma and Michel in \cite{MM} to treat   the H\"older regularity.  In the Sobolev category when $p>4$, this procedure at order $k-1$ is  meaningful and in the strong (continuous) sense  due to the Sobolev embedding theorem. Note that the top $k$-th order  derivatives are  still in the weak (distributional) sense where we need to use discretion. After a careful inspection of the post-surgical regularity on the pull-back of the data and push-forward of the solutions on the punctured bidisc, we utilize a weighted Hardy-type inequality to  obtain a sequence of refined  Sobolev estimates. These estimates eventually allow   the  weight loss from the singularity at $(0,0)$ to be precisely (and fortunately)  compensated by the weight gain from the truncation, so that the truncated solution enjoys the  (unweighted)   Sobolev regularity in Theorem \ref{main}. Throughout our proof, the assumptions $k\ge 1, p>4$ are crucial and repeatedly used. It is  not clear   whether the theorem still holds if $p\le 4$.   
\medskip

The organization of the paper is as follows. In Section 2, we give notations and preliminaries that are needed in the paper. In Section 3, we prove Theorem \ref{mainp} for the weighted Sobolev estimate on product domains, from which Corollary \ref{main4} follows. Section 4 is devoted to the proof of the main Theorem \ref{main} for the Sobolev estimate on the Hartogs triangle.  
 
\section{Notations and preliminaries}
 
\subsection{Weighted Sobolev spaces}

Denote by  $|S|$  the Lebesgue measure of a subset  $S$ in $\mathbb C^n$,  and $dV_{z_j}$  the volume integral element in the complex $z_j$ variable.  For  $z=(z_1, \cdots, z_n)\in \mathbb C^n$, let $\hat z_j =(z_1, \cdots, z_{j-1}, z_{j+1}, \cdots, z_n)\in \mathbb C^{n-1}$, where the $j$-th component of $z$ is skipped. Our weight space under consideration is as follows. 

\begin{definition}\label{aps}
Given $1<p<\infty$, a weight $\mu: \mathbb C^n\rightarrow [0, \infty)$ is said to be in $ A^*_p$ if the $A_p^*$ constant 
$$ A_p^*(\mu): = \sup \left(\frac{1}{|D|}\int_{D}\mu(z)dV_{z_j}\right)\left(\frac{1}{|D|}\int_{D} \mu(z)^{\frac{1}{1-p}}dV_{z_j}\right)^{p-1}<\infty, $$
 where the supremum is taken over  a.e. $\hat z_j\in \mathbb C^{n-1}, j=1, \ldots, n $, and   all discs $D\subset \mathbb C$. 
\end{definition}

When $n=1$, the $A_p^*$ space coincides with  the standard  Muckenhoupt's class $A_p$, the collection of all weights $\mu: \mathbb C^n\rightarrow [0, \infty)$ satisfying 
 \begin{equation*}
   A_p(\mu): =  \sup \left(\frac{1}{|B|}\int_{B}\mu(z)dV_z\right)\left(\frac{1}{|B|}\int_{B} \mu(z)^{\frac{1}{1-p}}dV_z\right)^{p-1}<\infty, \end{equation*}
where the supremum is taken over all balls $B\subset \mathbb C^n$. Clearly, $A_q\subset A_p$ if $1< q<p<\infty$. $A_p$ spaces also satisfy an open-end property:  if $\mu\in A_p$ for some $p>1$, then  $\mu\in A_{\tilde p} $ for some ${\tilde p}<p$.  See \cite[Chapter V]{Stein} for more details of the $A_p$ class. 

When $n\ge 2$, Definition \ref{aps} essentially says that  $\mu \in A_p^*$ if and only if the restriction of $\mu$ on any complex  one-dimensional slice $ \hat z_j$  belongs to $A_p$, with a uniform $A_p$ bound independent of $\hat z_j$. On the other hand, $\mu\in A^*_p$ if and only if the $\delta$-dilation $\mu_\delta(z): =\mu(\delta_1z_1, \ldots, \delta_n z_n)\in A_p$  with a uniform $A_p$ constant for all $\delta =(\delta_1, \ldots, \delta_n)\in  (\mathbb R^+)^n$ (see \cite[pp. 454]{GR}).   This in particular implies  $A^*_p\subset A_p$.    As will be seen in the rest of the paper,  the setting of $A^*_p$ weights allows us to apply the  slicing property  of product domains rather effectively. 
 
\medskip
Let $\Omega$ be a bounded domain in $\mathbb C^n$. Denote by $\mathbb Z^+$ the set of all positive integers. Given $k\in \mathbb Z^+\cup\{0\}, p\ge 1$, the weighted Sobolev space $W^{k, p}(\Omega, \mu)$ with respect to a weight $\mu\ge 0$ is the set of  functions   on $\Omega$ whose weak derivatives up to order $k$ exist and belong  to $L^p(\Omega, \mu)$. The corresponding  weighted  $W^{k, p}$ norm of a function  $h\in W^{k,p}(\Omega, \mu)$ is $$ \|h\|_{W^{k,p}(\Omega, \mu)}: = \left(\sum_{l=0}^k\int_\Omega |\nabla_z^lh(z)|^p\mu(z)dV_z\right)^\frac{1}{p}<\infty. $$
 Here  $\nabla_z^l h$  represents all $l$-th order weak derivatives of $h$. When $\mu\equiv 1$, $W^{k, p}(\Omega, \mu)$ is reduced to the (unweighted) Sobolev space   $W^{k,p}(\Omega)$.  As a direct consequence of the open-end property for $A_p$ and H\"older inequality, if $\mu\in A_p, p>1$, there exists some $q>1$ such that $ W^{k,p}(\Omega, \mu)\subset W^{k,q}(\Omega)$. 
 
 In the rest of the paper, for each $j = 1, \ldots, n$, we   use   $\nabla^{\alpha_j}_{z_j} h$  to  specify all $\alpha_j$-th order weak derivatives of $h$ in the complex $z_j$-th direction. For a multi-index $\alpha =(\alpha_1, \ldots, \alpha_n)$,  denote $ \nabla_{z_1}^{\alpha_1}\cdots \nabla_{z_n}^{\alpha_n}$ by $\nabla^\alpha_{z}$. Then for $l\in \mathbb Z^+$, $\nabla_z^{l} = \sum_{|\alpha|=l} \nabla^\alpha_{z} $.  We also represent the $\alpha_j$-th order derivative of $h$ with respect to the holomorphic $z_j$ and anti-holomorphic $\bar z_j$  variable by  $\partial^{\alpha_j}_{z_j} h$ and  $\bar\partial^{\alpha_j}_{z_j} h$, respectively.  When the context  is clear,  the letter $z$ may be dropped from those differential operators and we write instead $\nabla^{l},  \nabla^{\alpha_j}_j,  \nabla^\alpha, \partial^{\alpha_j}_{ j}$ and $  \bar\partial^{\alpha_j}_{j}$ etc.

\subsection{Weighted Sobolev estimates on planar domains}
Let $D$ be a   bounded  domain in $\mathbb C$ with Lipschitz boundary.   For $p>1, z\in D$, define
\begin{equation*} 
\begin{split}
 Th(z)&: =\frac{-1}{2\pi i}\int_D \frac{h(\zeta)}{\zeta- z}d\bar{\zeta}\wedge d\zeta, \ \ \text{for}\ \ h\in  L^p(D);\\
 Sh(z)&: =\frac{1}{2\pi i}\int_{bD}\frac{h(\zeta)}{\zeta- z}d\zeta, \ \ \text{for}\ \ h\in  L^p(bD). \end{split}
\end{equation*}
 Clearly, $ d\bar{\zeta}\wedge d\zeta = 2idV_\zeta$ in the above. $T$ and $S$ satisfy the  Cauchy-Green formula below: for any $h\in W^{1, p}(D), p>1$,
 $$ h  = Sh + T\bar\partial h\ \ \text{on} \ \ D $$
 in the sense of distributions.

The following weighted Sobolev regularity of $T$ and $S$ is essential in order to carry out the weighted Sobolev regularity of $\bar\partial$ on product domains.  It is worthwhile to note that \eqref{So} below fails if $k=0$, where  $S$ is not even well-defined. 



\begin{theorem}\cite{PZ2}\label{mainT}
Let $D\subset \mathbb C$ be a  bounded domain with $C^{k, 1}$ boundary and $\mu\in A_p, 1<p<\infty$.  For  $k\in \mathbb Z^+\cup\{0\}$, there exists a constant $C$ dependent only on $D, k$, $p$ and $  A_p(\mu)$, such that for all $h\in W^{k, p}(D, \mu)$,  
\begin{equation}\label{To}
      \|T h\|_{W^{k+1, p}(D, \mu)}\le C \|h\|_{W^{k, p}(D, \mu)}.
      \end{equation}
      If in addition $k\in \mathbb Z^+ $, then 
      \begin{equation}\label{So}
              \|S h\|_{W^{k, p}(D, \mu)}\le C \|h\|_{W^{k, p}(D, \mu)}.
      \end{equation}
\end{theorem}


\subsection{Product domains and the Hartogs triangle}
 A subset $\Omega\subset  \mathbb C^n$ is said  to be a product domain,   if    $\Omega = D_1\times\cdots\times D_n$, where each $D_j\subset \mathbb C, j=1, \ldots, n, $ is a  bounded domain in $\mathbb C$ such that  its boundary $bD_j$ consists of a finite number of  rectifiable Jordan curves which do not intersect one another. A product domain $\Omega$ is always  pseudoconvex,  and has   Lipschitz boundary if in addition each $bD_j$ is Lipschitz, $j=1, \ldots, n$. 
 
 Denote by $\triangle$ the unit disc in $\mathbb C$, and by  $\triangle^*: =\triangle\setminus \{0\}$    the punctured disc on $\mathbb C$. Then the punctured bidisc $ \triangle\times \triangle^*$ is biholomorphic to the Hartogs triangle $\mathbb H$ through the map  $\psi:  \triangle\times \triangle^* \rightarrow \mathbb H$, where \begin{equation}\label{psi}
    (w_1, w_2)\in \triangle\times \triangle^* \mapsto (z_1, z_2)= \psi(w)= (w_1w_2, w_2)\in \mathbb H.
\end{equation}  The inverse $\phi:  \mathbb H \rightarrow \triangle\times \triangle^*$ is given by \begin{equation}\label{phi}
    (z_1, z_2)\in \mathbb H  \mapsto (w_1, w_2) = \phi(z) = \left(\frac{z_1}{z_2}, z_2\right)\in \triangle\times \triangle^*.
\end{equation}
Note that $\mathbb H$ is not Lipschtiz near $(0,0)$.

It is well-known that any  domain with Lipschtiz boundary is a  uniform domain  (see \cite{GO} for the definition).  Recently, it was shown in  \cite[Theorem 2.12]{BFLS} that the Hartogs triangle  is also  a uniform domain. Thus according to  \cite{Jo}\cite[Theorem 1.1]{Ch}, both Lipschitz product domains and the Hartogs triangle   satisfy a weighted Sobolev extension property. Namely,  let $\Omega$ be either a  Lipschitz product domain or the Hartogs triangle. Then  for any weight $\mu\in A_p, 1< p<\infty, k\in \mathbb Z^+ $,  any  $h\in W^{k, p}(\Omega, \mu) $ can be extended as an element $\tilde h$ in $W^{k, p}(\mathbb C^n, \mu) $ such that 
  $$ \|\tilde h\|_{W^{k, p}(\mathbb C^n, \mu)}\le C\|h\|_{W^{k, p}(\Omega, \mu)}  $$
for some  constant $C$ dependent only on $k, p$ and the $A_p$ constant of $\mu$.  

For  simplicity of notations, throughout the rest of the paper, we shall say the two quantities $a$ and $b$ satisfy $a\lesssim b$, if  $a\le Cb$ for some  constant $C>0$ dependent only possibly on $\Omega, k, p$ and the $A_p^*$ constant $  A_p^*(\mu)$ (or $A_p(\mu)$).

\section{Weighted Sobolev estimates on product domains  }

Let $D_j\subset\mathbb C$, $j= 1, \ldots, n,$ be bounded domains with  $C^{k, 1}$ boundary,   $n\ge 2, k\in \mathbb Z^+\cup \{0\}$, and let $\Omega: = D_1\times\cdots\times D_n$. Denote by $T_j$ and $S_j$ the solid and boundary Cauchy integral operators $T$ and $S$ acting on functions along the $j$-th slice of $\Omega$, respectively. Namely, for $p>1, z\in \Omega$,
\begin{equation}\label{tj}
    \begin{split}
        &T_j h (z): =  \frac{-1}{2\pi i}\int_{D_j} \frac{h(z_1, \ldots, z_{j-1}, \zeta, z_{j+1}, \ldots, z_n)}{\zeta- z_j}d\bar{\zeta}\wedge d\zeta, \ \ \text{for}\ \ h\in  L^p(\Omega);\\
& S_j h (z): =  \frac{ 1}{2\pi i}\int_{bD_j} \frac{h(z_1, \ldots, z_{j-1}, \zeta, z_{j+1}, \ldots, z_n)}{\zeta- z_j}d\zeta, \ \ \text{for}\ \ h\in  L^p(b\Omega).
    \end{split}
\end{equation}

\begin{pro}\label{Tj}
Let $\Omega = D_1\times\cdots\times D_n$, where each $D_j  $  is a bounded domain in $\mathbb C$ with $C^{k, 1}$ boundary, $ k\in\mathbb Z^+\cup\{0\}$. Assume   $\mu\in A_p^*, 1<p<\infty$. Then for any $h\in W^{k, p}(\Omega, \mu)$,  
\begin{equation}\label{T_j}
          \|T_jh\|_{W^{k, p}(\Omega, \mu) }\lesssim \|h\|_{W^{k, p}(\Omega, \mu)}.\end{equation}
If in addition $k\in \mathbb Z^+$, then           \begin{equation}\label{S_j}
        \|S_jh\|_{W^{k-1, p}(\Omega, \mu) }\lesssim \|h\|_{W^{k, p}(\Omega, \mu)}.
    \end{equation}
\end{pro}

\begin{proof}
Without loss of generality, assume $j=1$ and $n=2$.  
For any multi-index $\alpha= (\alpha_1,  \alpha_2)$ with $|\alpha|\le  k$, since $\bar\partial_1 T_1 = id$,  we can further assume $\nabla^\alpha  T_1h = \partial_1^{\alpha_1} T_1 \left(\nabla_2^{\alpha_2}h\right)$. For a.e. fixed $z_2\in D_2$, $\mu(\cdot, z_2)\in A_p$ and $ \nabla_2^{
\alpha_2}h(\cdot, z_2)\in W^{\alpha_1, p}(D_1, \mu(\cdot, z_2)) $. Making use of \eqref{To}, we have
$$\int_{D_1}|\partial_1^{\alpha_1} T_1 \left(\nabla_2^{\alpha_2}h\right)(z_1, z_2) |^p\mu(z_1, z_2)dV_{z_1}\lesssim \sum_{l=0}^{\alpha_1} \int_{D_1}|\nabla_1^{l}\nabla_2^{\alpha_2}h(z_1, z_2) |^p\mu(z_1, z_2)dV_{z_1}. $$
Thus \begin{equation*}
    \begin{split}
 \|  \nabla^\alpha  T_1h\|^p_{ L^{p}(\Omega, \mu)} = &\int_{D_2}  \int_{D_1}|\partial_1^{\alpha_1} T_1 \left(\nabla_2^{\alpha_2}h\right)(z_1, z_2) |^p\mu(z_1, z_2)dV_{z_1}dV_{z_2}   \lesssim \|  h\|^p_{ W^{k, p}(\Omega, \mu)}.
    \end{split}
\end{equation*}

The boundedness of $S_1$ is proved similarly. Since $S_1h$ is holomorphic with respect to the $z_1$ variable, we only  consider $ \nabla^{\alpha} S_1h(z) = \partial_1^{ \alpha_1} S_1 (\nabla_2^{ \alpha_2} h)$ with $|\alpha|\le k-1$. Then $\nabla_2^{ \alpha_2} h(\cdot, z_2)\in W^{k- \alpha_2, p}(D_1)$ for a.e. $z_2\in D_2$. Noting that $k-\alpha_2\ge 1$, by  \eqref{So},
\begin{equation*} 
    \int_{D_1}|\partial_1^{\alpha_1} S_1 \left(\nabla_2^{\alpha_2}h\right)(z_1, z_2) |^p\mu(z_1, z_2)dV_{z_1}\lesssim   \sum_{l=0}^{\alpha_1+1}\int_{D_1}|\nabla_1^{l}\nabla_2^{\alpha_2}h(z_1, z_2) |^p\mu(z_1, z_2)dV_{z_1}. 
\end{equation*}
Here the sum for $l$ up to $\alpha_1+1$ is necessary in the case when  $ \alpha =(0, k-1)$, due to the absence of \eqref{So} at $k=0$ there. Hence $\|\nabla^{\alpha} S_1 h\|_{L^p(\Omega, \mu)}\lesssim \|h\|_{W^{k, p}(\Omega, \mu)}$.  

\end{proof}
\medskip

\begin{remark}\label{re}
a). The estimate  \eqref{T_j} is optimal. Indeed, consider $h(z_1, z_2)= |z_2|^{k-\frac{2}{p}}$ on ${\triangle\times \triangle}$. Then $h\in W^{k, s }({\triangle\times \triangle})$ for all $s<p$. However,  $T_1h(z_1, z_2) = \bar z_1|z_2|^{k-\frac{2}{p}} \notin W^{k, p}({\triangle\times \triangle})$.  \\
b). As a consequence of Theorem \ref{mainT},   one also has when $k\in \mathbb Z^+, 1<p<\infty, j=1, \ldots, n$,
\begin{equation}\label{T11}
               \sum_{l=0}^k\|\nabla_j^l T_jh\|_{L^{p}(\Omega, \mu) }\lesssim \sum_{l=0}^{k-1}\|\nabla_j^lh\|_{L^{ p}(\Omega, \mu)} \lesssim  \|h\|_{W^{k-1, p}(\Omega, \mu)},
      \end{equation}  
\begin{equation}\label{T112}
   \sum_{l=0}^k  \|\nabla_j^l T_jh\|_{W^{1, p}(\Omega, \mu) }\lesssim \|h\|_{W^{k, p}(\Omega, \mu)},
\end{equation}
and 
\begin{equation}\label{S11}
     \sum_{l=0}^k\|\nabla_j^l S_jh\|_{L^{p}(\Omega, \mu) }\lesssim \sum_{l=0}^{k}\|\nabla_j^lh\|_{L^{ p}(\Omega, \mu)} \lesssim \|h\|_{W^{k, p}(\Omega, \mu)}. 
\end{equation}
In the case when $\mu\equiv 1$ and $k=0$, an application of the classical complex analysis theory (see \cite{V} etc.) and Fubini theorem gives  for $1\le p< \infty$, \begin{equation}\label{T12}
 \|T_jh \|_{L^{p}(\Omega ) } \lesssim \|h\|_{L^{p}(\Omega)}. \ \ \
\end{equation} 
These inequalities will be used later.
\end{remark}
\medskip

Given a $(0, q)$ form \begin{equation*}
     \mathbf f = \sum_{\substack{ j_1<\cdots<j_q}}f_{\bar j_1\cdots\bar j_q}  d\bar z_{j_1}\wedge\cdots \wedge d\bar z_{j_q}\in C^1(\bar{\Omega}),
 \end{equation*}
define  $T_j \mathbf f$ and $S_j\mathbf f$ to be  the action on the corresponding  component functions. Namely,
\begin{equation*}
    \begin{split}
       & T_j\mathbf f: = \sum_{ \substack{ 1\le j_1<\cdots<j_q\le n}}T_jf_{ \bar j_1\cdots\bar j_q}  d\bar z_{j_1}\wedge\cdots \wedge d\bar z_{j_q};\\
       & S_j\mathbf f: = \sum_{\substack{ \\1\le j_1<\cdots<j_q\le n}}S_jf_{ \bar j_1\cdots\bar j_q}  d\bar z_{j_1}\wedge\cdots \wedge d\bar z_{j_q}.
    \end{split}
\end{equation*}
Furthermore,    define a projection  $\pi_k\mathbf f$  to be a $(0, q-1)$ form with
\begin{equation*}
    \pi_k\mathbf f: =  \sum_{\substack{  1\le k<j_2<\cdots<j_q\le n}}f_{ \bar k\bar j_2\cdots\bar j_q}  d\bar z_{j_2}\wedge\cdots \wedge d\bar z_{j_q}. 
\end{equation*}
 In their celebrated work \cite[pp. 430]{NW}, Nijenhuis and Woolf constructed  a solution operator of the $\bar\partial$ equation for $(0,q)$ forms on product domains. 

\begin{theorem}\cite{NW}\label{nw1}
Let $\Omega = D_1\times\cdots\times D_n$, where each $D_j  $  is a bounded domain in $\mathbb C$ with $C^{k, 1}$ boundary, $ k\in\mathbb Z^+$. If $\mathbf f\in  C^{1}(\bar{\Omega})$ is a $\bar\partial$-closed $(0, q)$ form on $\Omega$, then 
\begin{equation}\label{key}
  T\mathbf f: = 
  T_1\pi_1 \mathbf f +T_2S_1\pi_2 \mathbf f+\cdots+ T_nS_1\cdots S_{n-1}\pi_n\mathbf f
\end{equation}
is a solution to $\bar\partial   u = \mathbf f$ on $\Omega$. 
\end{theorem}

\begin{proof}[Proof of Theorem \ref{mainp}: ]  Given a  $\bar\partial$-closed $(0, q)$ form $\mathbf f\in W^{n-1,p}(\Omega, \mu), p>1$ (the  $k=1$ case in the theorem), we first verify that  $T\mathbf f$ in \eqref{key} is a weak solution to $\bar\partial u = \mathbf f$ on $\Omega$. Since $W^{n-1,p}(\Omega, \mu)\subset   W^{n-1,q}(\Omega)  $  for some $q>1$, for simplicity we directly assume $\mathbf f\in W^{n-1,p}(\Omega), p>1$.  Following an idea in \cite{PZ}, for each $j=1, \ldots, n$, let $ \{D^{(m)}_j\}_{m=1}^\infty$   be a family of strictly increasing open subsets  of $D_j$ such that\\
 a). for  $m\ge N_0\in \mathbb N$, $bD^{(m)}_j$ is  $C^{k, 1}$,  $\frac{1}{m+1}< dist(D^{(m)}_j, D_j^c)<\frac{1}{m}$;\\
 b). $H_j^{(m)}: \bar D_j\rightarrow \bar D_j^{(m)}$ is a $C^1$ diffeomorphism  with $\lim_{m\rightarrow \infty} \|H_j^{(m)}-id\|_{C^1(D_j)}=0$.

  Let $\Omega^{(m)}= D^{(m)}_1\times\cdots\times D^{(m)}_n$ be the product of those approximating  planar domains.  Denote by  $T^{(m)}_j, S^{(m)}_j$ and $T^{(m)}$ the operators defined in (\ref{tj}) and (\ref{key}) accordingly, with $\Omega$ replaced by $\Omega^{(m)}$. Then $T^{(m)} \mathbf f\in  W^{1, p}(\Omega^{(m)})$. Adopting the mollifier argument to $\mathbf f\in W^{n-1, p}(\Omega)$, we  obtain $\mathbf f^\epsilon\in C^1(\overline{\Omega^{(m)}})\cap W^{n-1, p}( \Omega^{(m)})$ such that  $$\|\mathbf f^\epsilon - \mathbf f\|_{W^{n-1, p}(\Omega^{(m)})}\rightarrow 0$$  as $\epsilon\rightarrow 0$ and  $\bar\partial \mathbf f^\epsilon =0$ on $\Omega^{(m)}$.

For each fixed $m$,  $ T^{(m)} \mathbf f^\epsilon\in W^{n-1, p}(\Omega^{(m)})$ when $\epsilon$ is small and $$\bar\partial T^{(m)} \mathbf f^\epsilon =\mathbf f^\epsilon \quad \text{in}\quad  \Omega^{(m)}$$
 by Theorem \ref{nw1}. Furthermore,   $$\|T^{(m)} \mathbf f^\epsilon - T^{(m)} \mathbf f\|_{W^{1,p}(\Omega^{(m)})} \lesssim \|\mathbf f^\epsilon - \mathbf f\|_{W^{n-1, p}(\Omega^{(m)})}\rightarrow 0$$ as $\epsilon\rightarrow 0$. In particular, $\lim_{\epsilon\rightarrow 0}T^{(m)} \mathbf f^\epsilon$ exists a.e. in $\Omega^{(m)}$
and is equal to $T^{(m)} \mathbf f \in W^{n-1, p}(\Omega^{(m)})$ pointwisely. 


Given a testing  form $\phi$ with a compact support  $K$, let $m_0\ge N_0$ be such that $K \subset \Omega^{(m_0-2)}$. 
Denote by  $\langle\cdot, \cdot\rangle_{\Omega}$ (and $\langle\cdot, \cdot\rangle_{\Omega^{(m_0)}}$) the inner product(s) in $L^2(\Omega)$ (and in $L^2({\Omega^{(m_0)}}$), respectively), and  $\bar\partial^*$ the formal adjoint of $\bar\partial$. For all $m\ge m_0$, one has
\begin{equation}\label{88}
 \langle T^{(m)}\mathbf f, \bar\partial^*\phi\rangle_{\Omega^{(m_0)}} =\lim_{\epsilon \rightarrow 0}\langle T^{(m)}\mathbf f^\epsilon, \bar\partial^*\phi\rangle_{\Omega^{(m_0)}}= \lim_{\epsilon \rightarrow 0}\langle \bar\partial T^{(m)}\mathbf f^\epsilon, \phi\rangle_{\Omega^{(m_0)}} = \lim_{\epsilon \rightarrow 0} \langle\mathbf f^\epsilon, \phi\rangle_{\Omega^{(m_0)}} = \langle\mathbf f, \phi\rangle_{\Omega}.
\end{equation}

We further show that \begin{equation}\label{99}
  \langle T\mathbf f, \bar\partial^*\phi\rangle_{\Omega}=\lim_{m\rightarrow \infty}\langle T^{(m)}\mathbf f, \bar\partial^*\phi\rangle_{\Omega^{(m_0)}}.
\end{equation}
    For simplicity, assume $\pi_j \mathbf f $ contains only one component function $f_j$, so does $\phi$. We will also drop various integral measures, which should be clear from the context. 
            For each $j=1, \ldots, n$, 
\begin{equation*}
  \begin{split}
    &\langle T_j^{(m)}S_1^{(m)}\cdots S_{j-1}^{(m)}\pi_j \mathbf f, \bar\partial^*\phi\rangle_{\Omega^{(m_0)}} \\
       = &\frac{1}{(2\pi i)^{j-1}}\int_{z\in K }T_j\left(\int_{\zeta_1\in b D_1^{(m)}}\cdots \int_{\zeta_{j-1}\in b D_{j-1}^{(m)}}\frac{f_j(\zeta_1, \cdots, \zeta_j, z_{j+1}, \cdots, z_n)\chi_{D_j^{(m)}}(\zeta_j)}{(\zeta_1-z_1)\cdots(\zeta_{j-1}-z_{j-1})}\right)\overline{\bar\partial^*\phi(z)}.
  \end{split}
\end{equation*}
Here $\chi_{D_j^{(m)}}$ is the characteristic function of $D_j^{(m)}\subset \mathbb C$. 

    For each $(z, \zeta_j)\in K\times D_j\setminus \{z_j=\zeta_j\}$, after a change of variables, there exists   some  function $h^{(m)}\in C (\bar D_1\times\cdots \bar D_{j-1})$, such that  $\|h^{(m)}-1\|_{C( D_1\times\cdots  D_{j-1} )}\rightarrow 0$ as $m\rightarrow \infty$ and 
\begin{equation*}
    \begin{split}
       &  \int_{\zeta_1\in b D_1^{(m)}}\cdots \int_{\zeta_{j-1}\in b D_{j-1}^{(m)}}\frac{f_j(\zeta_1, \cdots, \zeta_j, z_{j+1}, \cdots, z_n)\chi_{D_j^{(m)}}(\zeta_j) }{(\zeta_1-z_1)\cdots(\zeta_{j-1}-z_{j-1})} \\
       = &   \int_{\zeta_1\in b D_1 }\cdots \int_{\zeta_{j-1}\in b D_{j-1} } \frac{f_j(\zeta_1, \cdots, \zeta_j, z_{j+1}, \cdots, z_n)h^{(m)}(\zeta_1, \cdots, \zeta_{j-1})\chi_{D_j^{(m)}}(\zeta_j) }{(\zeta_1-z_1)\cdots(\zeta_{j-1}-z_{j-1})}.
    \end{split}
\end{equation*} 
   Notice that when  $z\in K(\subset \Omega^{(m_0-2)})$ and $\zeta_l\in b D_l^{(m)}, m\ge m_0, l=1, \ldots, j-1$,  $$\frac{1}{|\zeta_l-z_l|}\le \frac{1}{dist((\Omega^{(m)})^c, \Omega^{(m_0-2)})}\le\frac{1}{ dist((\Omega^{(m_0)})^c, \Omega^{(m_0-2)})}<  m_0^2 .$$
Hence  
\begin{equation}\label{00}
    \begin{split}
&\left|\langle T_j^{(m)}S_1^{(m)}\cdots S_{j-1}^{(m)}\pi_j \mathbf f, \bar\partial^*\phi\rangle_{\Omega^{(m_0)}} - \langle T_j S_1 \cdots S_{j-1} \pi_j \mathbf f, \bar\partial^*\phi\rangle_{\Omega^{(m_0)}} \right|      \\
\lesssim &\left\| T_j\left(\int_{  bD_1\times \cdots \times bD_{j-1}} \left|f_jh^{(m)} \chi_{ D_j^{(m)}}-f_j\right|\right) \right\|_{L^1(\Omega)}.
    \end{split}
\end{equation}
On the other hand,
\begin{equation}\label{9}
    \begin{split}
       & \left\|\int_{  bD_1\times \cdots \times bD_{j-1}} \left|f_jh^{(m)} \chi_{ D_j^{(m)}}-f_j\right|\right\|_{L^1(\Omega)}\\
       \lesssim &\left\| |f_j|\left|h^{(m)}\chi_{ D_j^{(m)}}-1\right|\right\|_{L^1(bD_1\times \cdots \times bD_{j-1}\times D_j\times\cdots\times D_n) }  \\
       \lesssim &\|f_j\|_{L^1(bD_1\times \cdots \times bD_{j-1}\times D_j\times\cdots\times D_n)} \|h^{(m)} -1\|_{C( D_1\times \cdots \times  D_{j-1} ) }\\
       &+\|f_j\|_{L^p(bD_1\times \cdots \times bD_{j-1}\times D_j\times\cdots\times D_n)}\text{vol}^{1-\frac{1}{p}}(D_j\setminus D_j^{(m)})  \\
        \lesssim &\|f_j\|_{W^{n-1, p}(\Omega)}\left(\|h^{(m)} -1\|_{C( D_1\times \cdots \times  D_{j-1} ) } +\text{vol}^{1-\frac{1}{p}}(D_j\setminus D_j^{(m)})\right) \rightarrow 0
    \end{split}
\end{equation} 
   as $m\rightarrow \infty$. Here we used the trace theorem in the third inequality. Combining \eqref{T12},   \eqref{00} and \eqref{9} we finally get   
   $$ \left|\left\langle T_j^{(m)}S_1^{(m)}\cdots S_{j-1}^{(m)}\pi_j \mathbf f -  T_j S_1 \cdots S_{j-1} \pi_j \mathbf f, \bar\partial^*\phi\right\rangle_{\Omega^{(m_0)}} \right|    \rightarrow 0 $$ 
 as $m\rightarrow \infty$. (\ref{99}) is  thus proved. Combining  (\ref{88}) with (\ref{99}),  we deduce that
\begin{equation*}
    \langle T\mathbf f, \bar\partial^*\phi\rangle_{\Omega}=\lim_{m\rightarrow \infty}\langle T^{(m)}\mathbf f, \bar\partial^*\phi\rangle_{\Omega^{(m_0)}} =\langle\mathbf f, \phi\rangle_\Omega,
\end{equation*}
which verifies   $T\mathbf f$ as a weak solution to $\bar\partial$ on $\Omega$.

We next prove the weighted Sobolev estimate for the operator $T$ defined in \eqref{key}. Since $\bar\partial T\mathbf f = \mathbf f$, we can further assume $\nabla^k=\partial^{\alpha} $   for any multi-index $\alpha= (\alpha_1, \ldots, \alpha_n)$ with $|\alpha|\le k$.  In view of   \eqref{key}   and the fact that $\pi_j$ being a projection is automatically  bounded in $W^{k, p}(\Omega, \mu)$, we only need to  estimate $\|\partial^\alpha T_n S_1\cdots S_{n-1} h\|_{L^{p}(\Omega, \mu)}$ in terms of $\|h\|_{W^{k+n-2, p}(\Omega, \mu)}$. Write   $  \partial^\alpha T_n S_1\cdots S_{n-1} h = (\partial_n^{\alpha_n} T_n) (\partial_1^{\alpha_1} S_1)\cdots (\partial_{n-1}^{\alpha_{n-1}} S_{n-1})   h  $. If $\alpha_n\ge 1$, we apply \eqref{T11} and    \eqref{S_j} inductively to have \begin{equation*}
    \begin{split}
        \| \partial^\alpha T_n S_1\cdots S_{n-1} h\|_{L^{p}(\Omega, \mu)} \lesssim  & \|  (\partial_1^{\alpha_1} S_1)\cdots (\partial_{n-1}^{\alpha_{n-1}} S_{n-1})    h\|_{W^{\alpha_n-1, p}(\Omega, \mu)}  \\
         \lesssim  & \|(\partial_2^{\alpha_2} S_2)   \cdots (\partial_{n-1}^{\alpha_{n-1}} S_{n-1})    h\|_{W^{\alpha_n-1+\alpha_1+1, p}(\Omega, \mu)}  \\
         \lesssim &\cdots\\
        \lesssim & \|h\|_{W^{\sum_{j=1}^n\alpha_j+n-2, p}(\Omega, \mu) } \le\|h\|_{W^{k+n-2, p}(\Omega, \mu) }. 
    \end{split}
\end{equation*}
If $\alpha_n=0$, then there exists some $1\le j\le n-1$, such that $\alpha_j\ge 1$. Without loss of generality, assume $\alpha_1\ge 1$. Then by \eqref{T11}, \eqref{S11} and \eqref{S_j} inductively,
\begin{equation*}
    \begin{split}
        \| \partial^\alpha T_n S_1\cdots S_{n-1} h\|_{L^{p}(\Omega, \mu)} \lesssim  & \|  (\partial_1^{\alpha_1} S_1)\cdots (\partial_{n-1}^{\alpha_{n-1}} S_{n-1})    h\|_{L^{ p}(\Omega, \mu)}  \\
         \lesssim  & \|(\partial_2^{\alpha_2} S_2)   \cdots (\partial_{n-1}^{\alpha_{n-1}} S_{n-1})    h\|_{W^{ \alpha_1, p}(\Omega, \mu)}  \\
         \lesssim  & \|(\partial_3^{\alpha_3} S_3)   \cdots (\partial_{n-1}^{\alpha_{n-1}} S_{n-1})    h\|_{W^{ \alpha_1 +\alpha_2+1, p}(\Omega, \mu)}  \\
         \lesssim &\cdots\\
        \lesssim & \|h\|_{W^{k+n-2, p}(\Omega, \mu) }. 
    \end{split}
\end{equation*}
The theorem is thus proved. 

\end{proof}



Similar to an example in \cite{Zhang2}, the following one shows that the $\bar\partial$ problem does not improve regularity in  weighted Sobolev spaces on product domains. As such the weighted Sobolev regularity obtained in Theorem  \ref{mainp} is optimal when $n=2$.

\begin{example}\label{ex2}
  For each  $k\in \mathbb Z^+, 1<p<\infty, \epsilon>0$ and any $s\in \left(\frac{2}{1+\epsilon}, 2\right)\setminus\{1\}$, consider $\mathbf f= (z_2-1)^{k-s}d\bar z_1 $ on ${\triangle\times \triangle}$, $\frac{1}{2}\pi <\arg (z_2-1)<\frac{3}{2}\pi$ and $\mu =|z_2-1|^{s(p-1)}$. Then $\mu\in A_p^*$, $\mathbf f\in W^{k, p}({\triangle\times \triangle}, \mu )$  and   is  $\bar\partial$-closed on ${\triangle\times \triangle}$. However, there does not exist a solution $u\in W^{k, p+\epsilon}({\triangle\times \triangle}, \mu)$ to $\bar\partial u =\mathbf f$ on ${\triangle\times \triangle}$.  \end{example}

\begin{proof}One can directly verify that $\mathbf f\in W^{k, p}({\triangle\times \triangle}, \mu) $ is $\bar\partial$-closed on ${\triangle\times \triangle}$ and  $\mu\in A_p^*$.  

Suppose there exists some $u\in W^{k, p+\epsilon}({\triangle\times \triangle}, \mu)$ satisfying $\bar\partial u =\mathbf f $ on ${\triangle\times \triangle}$. Then there exists some holomorphic function $h$ on ${\triangle\times \triangle} $, such that $u = (z_2-1)^{k-s}\bar z_1+h \in W^{k, p+\epsilon}({\triangle\times \triangle}, \mu)$. 
For each $(r, z_2) \in U: = (0,1) \times \triangle\subset \mathbb R^3$, consider  
    $$v(r, z_2): =\int_{|z_1|= r} {u}(z_1, z_2) dz_1. $$
   By H\"older inequality, Fubini theorem and the fact that $p>1$,   \begin{equation*}
        \begin{split}
            \|\partial_{z_2}^k v\|^{p+\epsilon}_{L^{p+\epsilon}(U, \mu)} =&\int_{U}  \left|\int_{|z_1|= r} \partial_{z_2}^k{u}(z_1, z_2) dz_1\right|^{p+\epsilon}\mu(z_2)dV_{z_2} dr\\
            =&  \int_{|z_2|<1}\int_{0}^1\left|r  \int_{0}^{2\pi} |\partial_{z_2}^k{u}(re^{i\theta}, z_2 )| d\theta  \right|^{p+\epsilon} dr\mu(z_2)dV_{z_2} \\
            \lesssim & \int_{|z_2|<1}\int_0^1 \int_{0}^{2\pi} |{\partial_{z_2}^k u}(re^{i\theta}, z_2 )|^{p+\epsilon}d\theta r dr \mu(z_2) dV_{z_2} \\
            = & \int_{|z_2|<1, |z_1|<1 } |\partial_{z_2}^k{u}(z )|^{p+\epsilon}\mu(z_2)dV_{z}\le \|{u}\|^{p+\epsilon}_{W^{k, p+\epsilon}({\triangle\times \triangle}, \mu)}<\infty.
        \end{split}
    \end{equation*} 
Thus  $\partial_{z_2}^k v\in L^{ p+\epsilon}(U, \mu)$. 

On the other hand, by Cauchy's theorem, for each $(r, z_2)\in U$,
  \begin{equation*}
    \begin{split}
    \partial_{z_2}^k v(r, z_2) =&(k-s)\cdots (1-s)\int_{|z_1|=r}  (z_2-1)^{-s}\bar z_1dz_1\\
    =&  (k-s)\cdots (1-s)(z_2-1)^{-s}\int_{|z_1|=r} \frac{r^2}{ z_1}dz_1 = 2(k-s)\cdots (1-s)\pi r^2i  (z_2-1)^{-s},
   \end{split}
  \end{equation*}
  which is not in  $L^{ p+\epsilon}(U, \mu)$ by the choice of $s>\frac{2}{1+\epsilon}$. This is a  contradiction!  
  
\end{proof}



Making use of Theorem \ref{mainp}, one can immediately prove the   weighted Sobolev estimate  for the
$\bar\partial$ problem on $\mathbb H$ in Corollary \ref{main4}. In comparison to the statement of Theorem \ref{main},  the solution operator in Corollary \ref{main4}  is the same for all Sobolev levels. 
\medskip

  \begin{proof}[Proof of Corollary \ref{main4}:]   For any $\mathbf f = \sum_{j=1}^2 f_j(z)d\bar z_j\in W^{k,p}(\mathbb H)  $,  making use of the change of variables formula we have the pull-back \begin{equation}\label{55}
     \psi^* \mathbf f = \bar w_2f_1\circ \psi  d\bar w_1 + \left(\bar w_1 f_1\circ \psi +f_2\circ \psi \right) d\bar w_2.   \end{equation} Moreover, noting by the chain rule 
     $$  \partial_{ w_1} = w_2\partial_{ z_1}, \ \ \  \partial_{ w_2} = w_1 \partial_{ z_1}+  \partial_{ z_2}, $$
we have $\psi^* \mathbf f\in  W^{k, p}({\triangle\times \triangle}, |w_2|^2)$ with
\begin{equation}\label{pull}
\begin{split}
     \|\psi^*\mathbf f\|^p_{W^{k, p}({\triangle\times \triangle}, |w_2|^{2}) }\lesssim& \sum_{j=1}^2\sum_{l=0}^k  \int_{{\triangle\times \triangle}} |\nabla_w^l (f_j\circ \psi)(w)|^p|w_2|^{2} dV_w\\
     \lesssim& \sum_{j=1}^2\sum_{l=0}^k \int_{\mathbb H} |\nabla_z^l f_j(z)|^pdV_z = \|\mathbf f\|^p_{W^{k,p}(\mathbb H) }. \end{split}\end{equation}

 Since $k\in \mathbb Z^+, p>2$, by  \eqref{pull} one has  $\psi^*\mathbf f$  to be $\bar\partial$-closed on ${\triangle\times \triangle}$ (see, for instance,  \cite[pp. 28]{Ma}). Making use of Theorem \ref{mainp}, there exists a solution $\tilde u \in W^{k, p }({\triangle\times \triangle}, |w_2|^2)$ solving $\bar\partial \tilde u =\psi^*\mathbf f $. Arguing in the same way as in  the proof of \cite[Theorem 1.2]{Zhang2}, we know that  $ \mathcal T\mathbf f: = \tilde u\circ \phi $ solves $ \bar\partial   u = \mathbf f$. 
Moreover, 
\begin{equation}\label{push}
\begin{split}
        \| \mathcal T\mathbf f \|^p_{ W^{k,p}(\mathbb H, |z_2|^{kp})}=& \sum_{l=0}^k\int_{\mathbb H}  |\nabla_z^l (\tilde u\circ \phi)(z) z_2^k|^{p}dV_z \\
        \lesssim &\sum_{l=0}^k \int_{{\triangle\times \triangle}} |\nabla_w^l\tilde u(w)|^p|w_2|^2dV_w =\| \tilde  u\|^p_{W^{k, p}({\triangle\times \triangle}, |w_2|^2) }.
    \end{split}
\end{equation}
 Here we used the chain rule 
 $$ \partial_{ z_1} = \frac{1}{z_2}\partial_{ w_1}, \ \ \ \partial_{ z_2} = -\frac{z_1}{z^2_2}\partial_{ w_1}+ \partial_{w_2}$$
 and the fact that $|z_1|<|z_2|$ on $\mathbb H$.
 
Finally, combining \eqref{pull}-\eqref{push} and Theorem \ref{mainp}, 
  $$\|\mathcal T\mathbf f\|_{W^{k,p}(\mathbb H, |z_2|^{kp})}\lesssim \|\tilde u\|_{W^{k, p}({\triangle\times \triangle}, |w_2|^2) }\lesssim \|\psi^*\mathbf f\|^p_{W^{k, p}({\triangle\times \triangle}, |w_2|^2) }\lesssim \|\mathbf f\|^p_{W^{k,p}(\mathbb H) }.   $$
  
  \end{proof}

\section{Optimal Sobolev regularity on the Hartogs triangle}

 In this section, following an idea of Ma and Michel in \cite{MM},  we shall adjust  the solution operator provided by  Corollary \ref{main4},    so that the new operator cancels the loss in the exponent of the weight. In detail, given a $W^{k, p}$ datum on the Hartogs triangle $\mathbb H$, we   truncate its $(k-1)$-th order Taylor polynomial at $(0,0)$ and then pull it back to the punctured bidisc $\triangle\times \triangle^*$. Upon extension and solving the $\bar\partial$ problem on the bidisc $\triangle\times \triangle$ using Theorem \ref{mainp}, we once again truncate the $(k-1)$-th order holomorphic Taylor polynomial in the $w_2$ variable at $w_2=0$ from the solution. Both  Taylor polynomials are meaningful when $p>4$ due to the Sobolev embedding theorem.  Moreover, we can obtain a refined weighted Sobolev regularity at each operation (Proposition \ref{tf} and Proposition \ref{soe}) as a consequence of the truncation.  Finally, pushing forward this truncated solution to $\mathbb H$, we show it is a solution to $\bar\partial$ on $\mathbb H$ that   maintains the same Sobolev regularity as that of the datum.
 
 Throughout the rest of the paper,    $z =(z_1, z_2)$ will serve as the variable  on $\mathbb H$, and $w =(w_1, w_2)$  as the variable on $\triangle\times \triangle$.    
  
\subsection{Truncating  data on the Hartogs triangle}
  Given a $\bar\partial$-closed $(0,1)$ form ${\mathbf  f}\in W^{k, p}(\mathbb H), k\in \mathbb Z^+,  p>4$, recalling  $\mathbb H$ satisfies the  Sobolev extension property, it extends to an element, still denoted by $\mathbf f$, in $ W^{k, p}(\mathbb C^2)$. In particular, by Sobolev embedding theorem, $\mathbf f\in C^{k-1, \alpha}(\mathbb H)$ for some $\alpha>0$. Denote by $\mathcal P_k  $   the $(k-1)$-th order Taylor polynomial operator  at $(0, 0)$. Namely, if  $h\in C^{k-1}$  near $(0, 0)$, then
  $$ \mathcal P_k h(z): = \sum_{l_1+l_2+s_1+s_2=0}^{k-1} \frac{\partial_{z_1}^{l_1}\bar\partial_{z_1}^{l_2}\partial_{z_2}^{s_1}\bar\partial_{z_2}^{s_2} h(0)}{l_1!l_2! s_1!s_2!}z_1^{l_1}\bar z_1^{l_2}z_2^{s_1}\bar z_2^{s_2}. $$
  Then $\mathcal P_k\mathbf f $ is $\bar\partial $-closed on $ \mathbb H$ and thus on ${\triangle\times \triangle}$ (see \cite[Lemma 3]{MM}). Applying the $W^{k, p}$ estimate of $\bar\partial$ on ${\triangle\times \triangle}$ (i.e., Theorem \ref{mainp} with $n=2$ and $ \mu \equiv 1$), one obtains some $u_{k}\in W^{k,p}({\triangle\times \triangle})$ satisfying  
\begin{equation}\label{pu}
\begin{split}
    &\ \ \ \ \ \ \ \ \ \ \ \ \ \ \ \ \ \ \ \ \ \ \ \  \bar\partial u_k = \mathcal P_{k}\mathbf f\ \ \text{on}\ \ {\triangle\times \triangle}; \\
    & \|u_{k}\|_{W^{k, p}({\triangle\times \triangle})}\lesssim \|\mathcal P_k\mathbf f\|_{W^{k, p}(\triangle\times \triangle)}\lesssim \|\mathcal P_k\mathbf f\|_{C^{k-1}(\mathbb H)}\lesssim \|\mathbf f\|_{W^{k, p}(\mathbb H)}.
\end{split}
 \end{equation}
Let $\psi$ be defined in \eqref{psi}. We truncate $  {\mathbf  f}$ by $\mathcal P_k  {\mathbf  f}$, and then  pull back the truncated datum  by $\psi$   to obtain $\psi^*(\mathbf f -\mathcal P_{k }\mathbf f)$ on  the punctured bidisc. 
 
\medskip

Denote by $\mathcal P_{2,k }$  the $(k -1)$-th order Taylor polynomial operator  in the complex $w_2$ variable  at $w_2=0$ of $C^{k-1}$ functions on $\triangle\times \triangle$. Then for any $h\in W^{k, p}(\mathbb H), k\in \mathbb Z^+, p>4$, 
\begin{equation*} 
    \psi^* \left(\mathcal P_k h\right) = \mathcal P_{2, k} \left(\psi^* h\right).
\end{equation*} 
 In particular,
\begin{equation}\label{ex1}
    \mathcal P_{2, k} \left(\psi^* (h-\mathcal P_k h )\right) =0.
\end{equation}
The following proposition states that the pull-back $\psi^*(\mathbf f -\mathcal P_{k }\mathbf f)$ of the truncated datum  satisfies   a  more refined Sobolev estimate than \eqref{pull}.

\begin{pro}\label{tf}
Let $\mathbf f\in W^{k, p}(\mathbb H)$ be a $\bar\partial$-closed $(0,1)$ form on $\mathbb H, k\in \mathbb Z^+, p>4$ and $\psi$ be defined in \eqref{psi}.  Let  $$\tilde {\mathbf f}= \tilde f_1d\bar w_1 +\tilde f_2 d\bar w_2: = \psi^*(\mathbf f -\mathcal P_{k }\mathbf f)  \ \ \text{on}\ \ {\triangle\times \triangle^*}.$$ 
Then $\tilde {\mathbf f}$  extends as  a $\bar\partial$-closed $(0,1)$ form  on ${\triangle\times \triangle}$, with $\tilde {\mathbf f}\in W^{k,p}(\triangle\times \triangle, |w_2|^2) $ and \begin{equation}\label{p2f}
    \mathcal P_{2, k} \tilde {\mathbf f} =  0.
    \end{equation} 
Moreover, for $t, s\in \mathbb Z^+\cup\{0\}, t+s\le k$, 
\begin{equation}\label{hh}
    \left\| |w_2|^{-k+s} \nabla^{t}_{w_1}\nabla^s_{w_2}\tilde {\mathbf  f}  \right\|_{L^p({\triangle\times \triangle}, |w_2|^2) }\lesssim \|\mathbf f\|_{W^{k, p}(\mathbb H) }.
\end{equation}
\end{pro}
\medskip

In order to prove Proposition \ref{tf}, we   need to establish a crucial weighted Hardy-type inequality on $\mathbb C$. We shall adopt the same notation $\mathcal P_k  $ for  the $(k-1)$-th order Taylor polynomial operator at $0$ on $C^{k-1}$ functions near $0\in \mathbb C$.  

\begin{lem}\label{tr1}
For any $h\in W^{k, p}(\mathbb C, |w|^2), k\in \mathbb Z^+, p>4$ with $\mathcal P_k h =0$,  and $j =0, \ldots, k$,
\begin{equation*} 
    \int_{\mathbb C}|\nabla_w^j  h(w)  |^p|w|^{2 -(k-j)p}dV_w\lesssim   \int_{\mathbb C}|\nabla^k_w h(w)|^p|w|^{2}dV_w. 
\end{equation*}
\end{lem}
 
\begin{proof}
Since the $j=k$ case  is trivial, we assume $j\le k-1$.  We shall show that for $h\in W^{l, p}(\mathbb C, |w|^2) $ with $\mathcal P_l h =0$, $l =1, \ldots,  k$, 
\begin{equation}\label{bb3}
    \int_{\mathbb C}|h(w) |^p|w|^{2 -lp}dV_w\lesssim   \int_{\mathbb C}|\nabla_w h(w)|^p|w|^{2-(l-1)p}dV_w. 
\end{equation}
If so,   replacing $l$ and  $h$ by $k-j$  and $\nabla_w^j h$ in \eqref{bb3}, respectively, then
$$  \int_{\mathbb C}|\nabla_w^j  h(w)  |^p|w|^{2 -(k-j)p}dV_w\lesssim   \int_{\mathbb C}|\nabla^{j+1}_w h(w)|^p|w|^{2-(k-j-1)p}dV_w.   $$
A standard induction on $j$ will complete the proof of the lemma. 

To show \eqref{bb3}, first apply the Stokes' theorem to $ |h(w)|^p|w|^{2-lp}\bar wd w$ on $\triangle_R\setminus \overline{\triangle_\epsilon}, 0<\epsilon<R<\infty$ to get
\begin{equation*}\label{st1}
    \begin{split}
   & \frac{1}{2i}\int_{b\triangle_R } |h(w)|^p|w|^{2-lp}\bar wd w -  \frac{1}{2i}\int_{b\triangle_\epsilon } |h(w)|^p|w|^{2-lp}\bar wd w\\
    =& \int_{ \triangle_R\setminus  \overline{\triangle_\epsilon} }  \bar\partial_{ w} \left( |h(w)|^p|w|^{2-lp}\bar w\right)dV_{w} \\
    =& \left(2-\frac{lp}{2}\right) \int_{\triangle_R\setminus  \overline{\triangle_\epsilon} }    |h(w)|^p|w|^{2-lp} dV_{w}+ \int_{\triangle_R\setminus  \overline{\triangle_\epsilon }}  \bar\partial_w \left( |h(w)|^p\right)|w|^{2-lp}\bar wdV_{w}.
\end{split}
\end{equation*} 
Since
$$\frac{1}{2i}\int_{b\triangle_R } |h(w)|^p|w|^{2-lp}\bar wd w  = \frac{1}{2}\int_0^{2\pi}|h(Re^{i\theta})|^pR^{4-lp} d\theta \ge 0,$$
one further has 
\begin{equation}\label{st}
\begin{split}
         \left(\frac{lp}{2}-2\right) \int_{\triangle_R\setminus \overline{ \triangle_\epsilon }  }  |h(w)|^p|w|^{2-lp} dV_{w} \le  \frac{1}{2i}\int_{b\triangle_\epsilon } |h(w)|^p|w|^{2-lp}\bar wd w  + \int_{\triangle_R\setminus  \overline{\triangle_\epsilon }}  \bar\partial_w \left( |h(w)|^p\right)|w|^{2-lp}\bar wdV_{w}.
\end{split}
  \end{equation}

 We claim that 
\begin{equation}\label{bg}
 \lim_{\epsilon\rightarrow 0} \epsilon^{3-lp} \int_{b\triangle_\epsilon } |h(w)|^pd\sigma_w  = 0, 
\end{equation}
which is equivalent to 
$$ \lim_{\epsilon\rightarrow 0}\left| \int_{b\triangle_\epsilon } |h(w)|^p|w|^{2-lp}\bar w d w\right| = 0. $$
Indeed, let $q$ be the dual of $p$, i.e.,  $\frac{1}{p}+\frac{1}{q}=1$.  For a.e.   $w\in b\triangle$ and  $0<\delta<\epsilon$,  applying Fubini theorem in polar coordinates, one can see $h(tw)\in W^{k, p}((\delta, \epsilon))$ as a function of $t$. By the fundamental theorem of calculus,  
we have
$$ h(\epsilon w) =  h(\delta w) + \int_\delta^\epsilon \frac{d}{dt} h(tw)dt.$$
Letting $\delta\rightarrow 0$ in the above, we have
$$ |h(\epsilon w)|\le   \int_0^\epsilon |\nabla h(tw)|dt.$$
An induction process further gives
\begin{equation}\label{poi}
    \begin{split}
        |h(\epsilon w) |^p\le &\left|\int_0^{\epsilon}\int_0^{t_1}\cdots \int_0^{t_{l-1}}|\nabla^l h(t_{l}w)|dt_{l}\cdots dt_2dt_1\right|^p\\
        \le &\left|\int_0^{\epsilon}\int_0^{{\epsilon}}\cdots \int_0^{{\epsilon}}|\nabla^l h(t_{l}w)|dt_{l}\cdots dt_2dt_1\right|^p\\
        \le& {\epsilon}^{(l-1)p}\left|\int_0^{\epsilon}|\nabla^l h(tw)|t^{\frac{3}{p}}\cdot t^{-\frac{3}{p}}dt\right|^p\\
        \le& {\epsilon}^{(l-1)p} \int_0^{\epsilon}| \nabla^l h(tw)|^pt^{3}dt \left(\int_0^{\epsilon} t^{-\frac{3q}{p}}dt)\right)^{\frac{p}{q}}\\
        \lesssim & {\epsilon}^{lp-4}\int_0^{\epsilon}| \nabla^l h(tw)|^pt^{3}dt.
    \end{split}
\end{equation}
Here we used the fact that  $-\frac{3q}{p}> -1$ when $p>4$ in the last inequality. Note that $$\epsilon\int_{b\triangle}|h({\epsilon}w) |^p  d\sigma_w = \int_{b\triangle_{\epsilon}}|h(w) |^p  d\sigma_w. $$ Multiplying both sides of \eqref{poi} by $ {\epsilon}^{4-lp}$ and integrating  over $b\triangle$, one has
\begin{equation*}
    \begin{split}
\epsilon^{3-lp} \int_{b\triangle_{\epsilon}}|h(w) |^p d\sigma_w\lesssim&     \int_0^{\epsilon}\int_{b\triangle}| \nabla^l h(tw)|^pt^{3}d\sigma_wdt\\
 =&   \int_0^{\epsilon}\int_{b\triangle_t}| \nabla^l h( w)|^p|w|^{2}d\sigma_w dt\\
 \le &  \int_{\triangle_\epsilon}| \nabla^l h(w)|^p|w|^{2}dV_w \rightarrow 0  
    \end{split}
\end{equation*}
 as $\epsilon\rightarrow 0$. The claim \eqref{bg} is thus proved. 
 
Pass  $\epsilon\rightarrow 0$ and $R\rightarrow \infty$ in \eqref{st}, and then make use of \eqref{bg}. Since  $  \frac{lp}{2}-2>0$, we    further infer
\begin{equation*}
    \begin{split}
      \int_{\mathbb C  }    |h(w)|^p|w|^{2-lp} dV_{w}\lesssim  &  \int_{\mathbb C  }    |\nabla_w h(w)||h(w)|^{p-1}|w|^{3-lp} dV_{w}\\
      =&   \int_{\mathbb C  }    |\nabla_w h(w)||w|^{\frac{2}{p} -(l-1)}\cdot |h(w)|^{p-1}|w|^{2-lp+l-\frac{2}{p}} dV_{w}\\
      \le& \left(\int_{\mathbb C  }    |\nabla_w h(w)|^p|w|^{2-(l-1)p}dV_w\right)^{\frac{1}{p}}\left(\int_{\mathbb C} |h(w)|^{p}|w|^{2-lp} dV_{w}\right)^{1-\frac{1}{p}}.
    \end{split}
\end{equation*}
 \eqref{bb3} follows by dividing both sides by $\left(\int_{\mathbb C} |h(w)|^{p}|w|^{2-lp} dV_{w}\right)^{1-\frac{1}{p}} $ and then taking the $p$-th power.

\end{proof}

\medskip

\begin{cor}\label{tr}
Let $D$ be a uniform domain in $\mathbb C$ and  $0\in D$. 
Then for any $h\in W^{k, p}(D, |w|^2), k\in \mathbb Z^+, p>4$ with $\mathcal P_k h =0$,  and $j =0, \ldots, k$,
\begin{equation*} 
    \int_{D}|\nabla_w^j  h(w)  |^p|w|^{2 -(k-j)p}dV_w\lesssim   \int_{D}|\nabla^k_w h(w)|^p|w|^{2}dV_w. \end{equation*}
\end{cor}

\begin{proof}
Given $h$ satisfying the assumption of the corollary, according to  \cite[Theorem 1.2]{Ch},  one can extend $h$ to be an element $\tilde h\in W^{k, p}(\mathbb C, |w|^2) $, such that  
$$ \int_{\mathbb C}|\nabla^k_w \tilde h(w)|^p|w|^{2}dV_w \lesssim \int_{D}|\nabla^k_w h(w)|^p|w|^{2}dV_w.$$
Obviously $\mathcal P_k\tilde h =0$. Hence making use of Lemma \ref{tr1} to $\tilde h$, we have 
\begin{equation*}
    \begin{split}
        \int_{D}|\nabla_w^j  h(w)  |^p|w|^{2 -(k-j)p}dV_w\le &  \int_{\mathbb C}|\nabla_w^j \tilde h(w)  |^p|w|^{2 -(k-j)p}dV_w\\
        \lesssim& \int_{\mathbb C}|\nabla^k_w \tilde h(w)|^p|w|^{2}dV_w \lesssim \int_{D}|\nabla^k_w h(w)|^p|w|^{2}dV_w. 
    \end{split}
\end{equation*}
\end{proof}

\begin{remark}\label{re2} 
Recall that any domain with Lipschitz boundary is a uniform domain. As a direct consequence of Corollary \ref{tr}, whenever $h\in W^{k, p}(\triangle, |w|^2), p>4$  with $\mathcal P_k h =0$, then $w^{-k}h\in L^p(\triangle, |w|^2).$
\end{remark}

 As shown in the proof of Lemma \ref{tr1} (and thus Corollary \ref{tr}), the assumption $p>4$ is essential and can not be dropped. Now we are ready to prove Proposition \ref{tf} making use of Corollary \ref{tr}.

\begin{proof}[Proof of Proposition \ref{tf}: ]  
The $\bar\partial$-closedness of $\psi^*\mathbf f$   on ${\triangle\times \triangle}$ was checked in Corollary \ref{main4}. Thus $ \tilde{\mathbf f}$   is $\bar\partial$-closed on ${\triangle\times \triangle}$, and by \eqref{55}, \begin{equation}\label{ho}
    \tilde f_1=  \bar w_2\psi^*(f_1 - \mathcal P_k f_1), \ \ \tilde f_2= \bar w_1\psi^*(f_1 -\mathcal P_k f_1)   +\psi^*(f_2-\mathcal P_kf_2).
\end{equation}
\eqref{p2f} follows from the above \eqref{ho} and  \eqref{ex1}.

Next we prove \eqref{hh}. For $l_1, l_2\in \mathbb Z^+\cup\{0\}$ with $l_1+l_2 = t$,  
$$   \bar\partial^{l_1}_{w_1}\partial^{l_2}_{w_1} \left(\psi^* f_j\right) =  \bar w_2^{l_1} w_2^{l_2}  \psi^* \left( \bar\partial^{l_1}_{z_1}\partial^{l_2}_{z_1} f_j\right), \ \ \  j= 1, 2.$$
 Observing that  $$ \nabla_{z_1}^t( \mathcal P_kf_j) =  \mathcal P_{k-t}\left( \nabla_{z_1}^t f_j\right), $$ we get from \eqref{ho} that
\begin{equation*}
    \begin{split}
        &\left\|  |w_2|^{-k+s} \nabla^{t}_{w_1}\nabla^s_{w_2}\tilde {\mathbf  f} \right\|_{L^p({\triangle\times \triangle}, |w_2|^2) }\\
        \lesssim& \sum_{j=1}^2  \left\| |w_2|^{-k+s}  \nabla_{ w_2}^{s} \nabla_{ w_1}^{t}\left( \psi^*(f_j-\mathcal P_kf_j) \right)\right\|_{L^p({\triangle\times \triangle}, |w_2|^2) }\\
        \lesssim&\sum_{j=1}^2 \sum_{l_1+l_2=t} \left\| |w_2|^{-k+s}  \nabla_{ w_2}^{s} \left( \bar w_2^{l_1}w_2^{l_2} \psi^*\left(\nabla_{ z_1}^t  f_j-\mathcal P_{k-t}\left( \nabla_{z_1}^t  f_j\right)\right)\right)\right\|_{L^p({\triangle\times \triangle}, |w_2|^2) }\\
        \lesssim &\sum_{1\le j\le 2}\sum_{ 0\le l\le s}  \left\| |w_2|^{-k  +t+l}  \nabla_{ w_2}^{l} \left({ \psi^*}\left( \nabla_{ z_1}^t f_j-\mathcal P_{k-t}\left( \nabla_{ z_1}^t f_j \right)\right)\right)\right\|_{L^p({\triangle\times \triangle}, |w_2|^2) }.
    \end{split}
\end{equation*} 
 Thus we only need to estimate $ \left\|  |w_2|^{-k  +t+l}  \nabla_{ w_2}^{l} \left({ \psi^*}\left( \nabla_{ z_1}^t f_j-\mathcal P_{k-t}\left( \nabla_{ z_1}^t f_j \right)\right)\right)\right\|_{L^p({\triangle\times \triangle}, |w_2|^2) },  0\le l\le s.$


  For each fixed  $w_1\in \triangle$, let $$h_{w_1} : =  \psi^*\left( \nabla_{ z_1}^t f_j -\mathcal P_{k-t}\left(\nabla_{ z_1}^t  f_j\right)  \right)(w_1, \cdot) .$$
Then  $\mathcal P_{k-t}h_{w_1} =0$  by \eqref{ex1}.  Applying Corollary \ref{tr} to $h_{w_1}$ on $\triangle$, we have for $ 0\le l(\le s)\le k-t$,
  \begin{equation}\label{cc}
      \begin{split}
           & \left\| |w_2|^{-k  +t+l}  \nabla_{ w_2}^{l} \left({ \psi^*}\left(\nabla_{ z_1}^t  f_j-\mathcal P_{k-t}\left(\nabla_{ z_1}^t  f_j \right)\right)\right)\right\|^p_{L^p({\triangle\times \triangle}, |w_2|^2) }\\
  \le &  \int_{\triangle }\int_{\triangle}|w_2|^{2 -(k-t-l)p}\left|  \nabla_{ w_2}^{l} \left(\psi^*\left( \nabla_{ z_1}^t  f_j  -\mathcal P_{  k-t} \left(\nabla_{ z_1}^t  f_j \right)  \right)\right)(w_1, w_2)\right|^pdV_{w_2}dV_{w_1}\\
  \lesssim &  \int_{\triangle}\int_{\triangle}|w_2|^{ 2}\left|  \nabla_{ w_2}^{k-t }\left(\psi^*\left( \nabla_{ z_1}^t  f_j -\mathcal P_{k-t}\left( \nabla_{ z_1}^t f_j\right)  \right)\right)(w_1, w_2)\right|^pdV_{w_2}dV_{w_1}.
      \end{split}
  \end{equation}
  
On the other hand,  note that  for any function $h    \in W^{k-t, p}(\mathbb H)$, $l_1+l_2 = k-t$,
$$ \bar\partial_{ w_2}^{l_1} \partial_{ w_2}^{l_2} \psi^* h = \sum_{m_1=0}^{l_1}\sum_{m_2=0}^{l_2} C_{m_1, m_2, l_1, l_2}\bar w_1^{m_1} w_1^{m_2}\psi^*\left( \bar\partial_{ z_1}^{m_1}\bar\partial_{ z_2}^{l_1-m_1}\partial_{ z_1}^{m_2} \partial_{ z_2}^{l_2-m_2}  h\right)$$
for  some constants $C_{m_1, m_2, l_1, l_2}$ dependent only on $m_1, m_2, l_1$ and $l_2$. Thus
\begin{equation*}
    \begin{split}
        & \left|\nabla_{ w_2}^{k-t }\left(\psi^*\left( \nabla_{ z_1}^t  f_j -\mathcal P_{k-t}\left( \nabla_{ z_1}^t f_j\right)  \right) \right) \right|\\\lesssim  &  \sum_{m=0}^{k-t}  |w_1|^{m}\left|\psi^*\left( \nabla_{ z_1}^{t+m} \nabla_{z_2}^{k-t-m}  f_j\right) - \psi^*\left( \nabla_{ z_1}^{m}\nabla_{ z_2}^{k-t-m} \left(\mathcal P_{k-t}\left( \nabla_{ z_1}^t f_j\right) \right)\right) \right|\\
         \le& \sum_{m=0}^{k-t} \left|\psi^*\left( \nabla_{ z_1}^{t+m} \nabla_{z_2}^{k-t-m}  f_j\right)\right|.
    \end{split}
\end{equation*} 
Here we used in the last equality the fact  that 
$\nabla_z^{ k-t} \left(\mathcal P_{k-t}\left( \nabla_{ z_1}^t f_j\right) \right) =0.$
Hence by a change of variables \eqref{cc} is further    estimated   as follows. 
\begin{equation*} 
      \begin{split}
           & \left\| |w_2|^{-k  +t+l}  \nabla_{ w_2}^{l} \left({ \psi^*}\left( \nabla_{ z_1}^t f_j-\mathcal P_{k-t}\left(\nabla_{z_1}^t  f_j\right)\right)\right)\right\|^p_{L^p({\triangle\times \triangle}, |w_2|^2) }\\
           \lesssim&  \sum_{m=0}^{k-t}\int_{\triangle}\int_{\triangle}|w_2|^{ 2}\left| \psi^*\left( \nabla_{z_1}^{t+m}\nabla_{ z_2}^{k-t-m} f_j \right) (w_1, w_2)\right|^pdV_{w_2}dV_{w_1}\\
           \lesssim & \|\psi^*(\nabla_z^k f_j) \|^p_{L^p(\triangle\times \triangle, |w_2|^2)} \lesssim \|\nabla_z^k f_j\|^p_{L^{p}(\mathbb H)} \le \|\mathbf f\|^p_{W^{k, p}(\mathbb H)}.
\end{split}
\end{equation*}
The proof of \eqref{hh} is complete. That $\tilde {\mathbf f}\in W^{k,p}(\triangle\times \triangle, |w_2|^2) $
is a direct consequence of \eqref{hh}.

\end{proof}

\subsection{Truncating solutions on the bidisc}
Given $\tilde {\mathbf f}$ in Proposition \ref{tf}, let $u^*$ be the  solution to $\bar\partial u^* = \tilde {\mathbf f}$ on ${\triangle\times \triangle}$ obtained in   Theorem \ref{mainp} with
\begin{equation}\label{t2}
    \|u^*\|_{W^{k, p}({\triangle\times \triangle}, |w_2|^2) }\lesssim \|\tilde {\mathbf f}\|_{W^{k, p}({\triangle\times \triangle}, |w_2|^2) }.
\end{equation} Consider
\begin{equation}\label{tud}
\begin{split}
       \tilde u(w_1, w_2): =&u^*(w_1, w_2) -\tilde{\mathcal P}_{2, k} u^*(w_1, w_2)\\
   =& u^*(w_1, w_2) -\sum_{l=0}^{k-1}\frac{1}{l!}w_2^l\partial_{w_2}^l u^*(w_1, 0),\ \ (w_1, w_2)\in {\triangle\times \triangle},  
\end{split}
\end{equation}   
 where $\tilde{\mathcal P}_{2, k}$ is the $(k-1)$-th order  holomorphic Taylor polynomial operator in the $w_2$ variable at $w_2=0$.   $\tilde u$ is well defined, due to the facts that for each fixed $w_1\in \triangle$, $l\le k-1$,  $\partial_{w_2}^l u^*(w_1, \cdot) \in W^{1, p}(\triangle, |w_2|^2)$, and when $p>4$, \begin{equation}\label{em}
      W^{1, p}(\triangle, |w_2|^2) \subset W^{1, q}(\triangle)\subset C^\alpha(\triangle) \end{equation} for some $q>2$, and $\alpha = 1-\frac{2}{q}$. Here the last inclusion $W^{1, q}(\triangle)\subset C^\alpha(\triangle)$ is the Sobolev embedding theorem; the  inclusion $W^{1, p}(\triangle, |w_2|^2) \subset W^{1, q}(\triangle)$ can be seen as follows. Choose some  $r\in (\frac{2}{p}, \frac{1}{2})$ and let $q = pr$. Then $q>2$ and $ \frac{r}{1-r}<1$. For any $h\in W^{1, p}(\triangle, |w_2|^2)$,
$$\int_\triangle |h(w)|^q dV_w =  \int_\triangle |h(w)|^q |w_2|^{2r} |w_2|^{-2r}dV \le \left(\int_\triangle |h(w)|^q |w_2|^2dV\right)^{r}\left(\int_\triangle |w_2|^{-\frac{2r}{1-r}} dV_w\right)^{1-r} <\infty,$$ 
 and similarly $|\nabla h| \in L^{q}(\triangle) $.  
 The goal of this subsection is to show  that $\tilde u$ satisfies the following refined weighted estimate. 
\medskip 

\begin{pro}\label{soe}Let $\tilde u$ be defined in \eqref{tud}. Then $\tilde u \in  W^{k, p}({\triangle\times \triangle}, |w_2|^2)$. Moreover, for each $s, t\in \mathbb Z^+\cup\{0\}$ with $s+t\le k,$ we have
\begin{equation*}
     \left\||w_2|^{-k+s} \partial_{ w_1}^{t}\partial_{ w_2}^{s}\tilde  u \right\|_{L^p({\triangle\times \triangle}, |w_2|^2)}\lesssim \left\| \mathbf f \right\|_{  W^{k,p}(\mathbb H)}. 
\end{equation*}

\end{pro}
 \medskip

 We begin by first   proving $\tilde u \in W^{k, p}({\triangle\times \triangle}, |w_2|^2)$ below. It is worth pointing out that, 
 arguing similarly as in \eqref{em}, one has when $k\in \mathbb Z^+, p>4$, 
 $$ W^{k, p}(\triangle, |w_2|^2) \subset W^{k, q}(\triangle)\subset C^{k-1, \alpha}(\triangle). $$
 for some $q>2$  and $\alpha>0$. In particular, for any $h\in W^{k, p}(\triangle\times \triangle, |w_2|^2), k\in \mathbb Z^+, p>4$, we have $h(w_1, \cdot)\in C^{k-1, \alpha}(\triangle)$ for a.e. fixed $w_1\in \triangle$.
 
 \medskip

\begin{lem}\label{tu} Let  $\tilde u$ be defined in \eqref{tud}. For each $l=0, \ldots,  k-1$, $\partial_{w_2}^l u^*(w_1, 0)\in W^{k, p}({\triangle\times \triangle}, |w_2|^2)$ with
\begin{equation}\label{tuu}
     \left\|\partial_{w_2}^l u^*(w_1, 0)\right\|_{W^{k,p}({\triangle\times \triangle}, |w_2|^2)}\lesssim \left\| \mathbf f \right\|_{  W^{k,p}(\mathbb H)}. 
\end{equation}
Consequently, $\tilde u\in W^{k, p}({\triangle\times \triangle}, |w_2|^2)$ satisfying 
\begin{equation}\label{p2}
    \mathcal P_{2, k} \tilde u =  0,
    \end{equation}
\begin{equation}\label{hot}
    \bar\partial \tilde u = \tilde {\mathbf f} \ \ \text{on}\ \ {\triangle\times \triangle}
\end{equation} and 
\begin{equation}\label{tuf}
    \|\tilde u\|_{W^{k, p}({\triangle\times \triangle}, |w_2|^2) }\lesssim \|  {\mathbf f}\|_{W^{k, p}(\mathbb H) }.
\end{equation}
\end{lem}

\begin{proof}

We first show that $ \sum_{l=0}^{k-1}w_2^l\partial_{w_2}^l u^*(w_1, 0) $ is holomorphic on ${\triangle\times \triangle}$, from which \eqref{hot} follows. Clearly, it is holomorphic in the $w_2$ variable. For the holomorphy in the $w_1$ variable, note that     $$ \bar\partial_{ w_1}\partial_{w_2}^l u^* = \partial_{w_2}^l \tilde f_1$$ in the weak sense.  On the other hand, for   fixed $w_1\in \triangle$, $\partial_{w_2}^l\tilde f_1(w_1, \cdot)\in C^\alpha(\triangle)$ for some $\alpha>0$ by \eqref{em}, and   $ \partial_{w_2}^l\tilde f_1(w_1, 0) =0$ by   \eqref{p2f}. Thus  $\bar\partial_{ w_1}\partial_{w_2}^l u^*\in C^\alpha(\triangle)$ with  $\bar\partial_{ w_1}\partial_{w_2}^l u^*(w_1, 0) =0$.

Next we prove  \eqref{tuu}. By the holomorphy of $\partial_{w_2}^l u^*(w_1, 0)$ above, it suffices to  estimate $ \left\| \partial_{w_1}^{t}\partial_{w_2}^{l} u^*(w_1, 0)\right\|_{L^p({\triangle\times \triangle}, |w_2|^2)}$ for $t =0,\ldots, k$ and $ l=0, \ldots, k-1$.
Let $\chi$ be a smooth function on $\triangle$ such that $\chi=1$ in $\triangle_{\frac{1}{2}}$ and $\chi =0$ outside $\triangle$. By \eqref{key} (or directly verifying  $u^* = T_1\tilde f_1 + T_2S_1 \tilde f_2 = T_2\tilde f_2 +T_1S_2\tilde f_1 $), we have
\begin{equation*}
    \begin{split}
 \partial_{w_1}^{t}\partial_{ w_2}^{l} u^* =& \partial_{w_2}^{l}T_2\left((1-\chi(w_2))  \partial_{ w_1}^{t}\tilde f_2\right) +  \partial_{ w_1}^{t}\partial_{w_2}^{l}T_2\left( \chi(w_2)  \tilde f_2\right) + \partial_{ w_2}^{l}S_2 \left(\partial_{w_1}^{t}   T_1 \tilde f_1\right)\\
 =&: A_1+A_2+A_3.        
    \end{split}
\end{equation*}

For $A_3$, let  $h: =\partial_{w_1}^{t}   T_1 \tilde f_1$. Since $ t\le k$,  by \eqref{T112} $h \in W^{1, p}({\triangle\times \triangle}, |w_2|^2)$, with $\|h\|_{W^{1, p}({\triangle\times \triangle}, |w_2|^2)}\lesssim \|\tilde f_1\|_{W^{k, p}({\triangle\times \triangle}, |w_2|^2)}$. Note that for   $w_1\in \triangle$, $$   A_3(w_1, 0) = \frac{ l!}{2\pi i} \int_{b\triangle}\frac{   h (w_1, \zeta)}{\zeta^{l+1}}d\zeta.$$
Hence 
\begin{equation}\label{dd}
\begin{split}
      \|A_3(w_1, 0)\|^p_{L^{p}({\triangle\times \triangle}, |w_2|^2)}\lesssim &\int_\triangle\left|\int_{b\triangle}|h(w_1, \zeta)|d\sigma_{\zeta} \right|^p dV_{w_1}\int_\triangle|w_2|^2dV_{w_2}\\
      \lesssim &\int_\triangle \left|\int_{\triangle}  |h(w_1, w_2)| +|\nabla_{w_2} h(w_1, w_2)| dV_{w_2}\right|^p   dV_{w_1}\\
      \lesssim& \|h\|^p_{W^{1, p}({\triangle\times \triangle}, |w_2|^2)} 
      \lesssim   \|\tilde f_1\|^p_{W^{k, p}({\triangle\times \triangle}, |w_2|^2)}\\ \lesssim & \|  {\mathbf f}\|^p_{W^{k, p}(\mathbb H) }. 
\end{split}
 \end{equation}
Here in the second line we used  the trace theorem for $W^{1,1}(\triangle)\subset L^1(\partial \triangle)$; in the third line we used H\"older inequality and the fact that $|w_2|^2\in A_p$ (or directly that $|w_2|^{-\frac{2}{p-1}}\in L^1(\triangle)$); in the fourth line we used Proposition \ref{tf}.

For $A_1$, by the choice of $\chi$, we have 
$$   A_1(w_1, 0)  =-\frac{l!}{2\pi i} \int_{\triangle}\frac{  (1-\chi(\zeta)) \partial_{w_1}^{t} \tilde f_2   (w_1, \zeta) }{\zeta^{l+1} }d\bar\zeta\wedge d\zeta,$$
with $\left| \frac{1-\chi(\zeta)}{\zeta^{l+1}}\right|\lesssim 1$ on $\triangle$. Thus by  Proposition \ref{tf} and the fact that $|w_2|^2\in A_p$ similarly,
\begin{equation}\label{aa}
\begin{split}
     \left\|  A_1(w_1, 0)\right\|^p_{L^p({\triangle\times \triangle}, |w_2|^2)}\lesssim & \int_\triangle\left|\int_\triangle \left| \partial_{w_1}^{t} \tilde f_2   (w_1, \zeta) \right|dV_{\zeta}\right|^p dV_{w_1}\int_\triangle|w_2|^2dV_{w_2}\\
     \lesssim& \int_\triangle\left|\int_\triangle \left| \partial_{w_1}^{t} \tilde f_2   (w_1, w_2) \right|dV_{w_2}\right|^p dV_{w_1}\\
     \lesssim & \int_\triangle\int_\triangle \left| \partial_{w_1}^{t} \tilde f_2   (w_1, w_2) \right|^p|w_2|^2dV_{w_2}  dV_{w_1}\\
     \le & \|\tilde {\mathbf f}\|^p_{W^{k, p}({\triangle\times \triangle}, |w_2|^2) }\lesssim \|  {\mathbf f}\|^p_{W^{k, p}(\mathbb H) }. 
\end{split}
   \end{equation}

Now we treat $A_2$. With a change of variables, rewrite it as\begin{equation*}
    \begin{split}
       A_2(w_1, 0) =& \left.-\frac{1}{2\pi i}\partial_{ w_1}^{t} \partial_{w_2}^{l} \int_{\mathbb C}\frac{ \chi(\zeta+w_2)  \tilde f_2 (w_1, \zeta+w_2)}{\zeta}d\bar\zeta\wedge d\zeta\right|_{w_2=0}\\
        =&\left.-\frac{1}{2\pi i}\partial_{ w_1}^{t}\int_{\mathbb C}\frac{\partial_{\zeta}^{l}\left( \chi(\zeta+w_2)   \tilde f_2 (w_1, \zeta+w_2)\right)}{\zeta}d\bar\zeta\wedge d\zeta\right|_{w_2=0}\\
         =& -\frac{1}{2\pi i}\partial_{ w_1}^{t}\int_{\mathbb C}\frac{\partial_{\zeta}^{l}\left( \chi(\zeta )   \tilde f_2 (w_1, \zeta )\right)}{\zeta}d\bar\zeta\wedge d\zeta . 
    \end{split}
\end{equation*}
 Note that $\chi(\cdot)\tilde f_2(w_1, \cdot)\in C_c^{k-1, \alpha}(\triangle) $ for some $\alpha>0$ with $\mathcal P_k\left(\chi(\cdot)\tilde f_2(w_1, \cdot)\right) =0$. In particular, for $j=0, \ldots, l$, $ \left|\partial_{\zeta}^{j}\left( \chi(\zeta )   \tilde f_2 (w_1, \zeta )\right)\right|\lesssim |\zeta|^{k-1-j+\alpha}$ near $0$.   With a repeated application of  Stokes' theorem, we have  
\begin{equation*}
    \begin{split}
    A_2(w_1, 0)    =&-\frac{l!}{2\pi i}   \partial_{ w_1}^{t}  \int_{\mathbb C}\frac{    \chi(\zeta )\tilde f_2 (w_1, \zeta ) }{\zeta^{l+1}}d\bar\zeta\wedge d\zeta\\
   =&-\frac{l!}{2\pi i} \int_{\triangle}\frac{    \chi(\zeta )  \partial_{ w_1}^{t}  \tilde f_2   (w_1, \zeta ) }{\zeta^{l+1}}d\bar\zeta\wedge d\zeta.      
    \end{split}
\end{equation*}
Since $l\le k-1$, making use of Proposition \ref{tf} with $s=0$ and the fact that $|w_2|^2\in A_p$ again, we get
\begin{equation}\label{ee}
    \begin{split}
        \|A_2(w_1, 0)\|^p_{L^p({\triangle\times \triangle}, |w_2|^2)}\lesssim & \int_\triangle \left|\int_{\triangle} |\zeta|^{-(l+1)} \left|\partial_{w_1}^t\tilde f_2(w_1, \zeta)\right|dV_{\zeta}\right|^p  dV_{w_1}\int_\triangle|w_2|^2dV_{w_2}\\
        \lesssim& \int_\triangle \left|\int_{\triangle} |w_2|^{-(l+1)} \left|\partial_{w_1}^t\tilde f_2(w_1, w_2)\right|dV_{w_2}\right|^p  dV_{w_1}\\
        \lesssim& \int_\triangle  \int_{\triangle} |w_2|^{-(l+1)p} \left|\partial_{w_1}^t\tilde f_2(w_1, w_2)\right|^p |w_2|^2dV_{w_2} dV_{w_1}\\
        \lesssim&  \left\| |w_2|^{-k}\partial_{w_1}^t\tilde f_2\right\|^p_{L^p({\triangle\times \triangle}, |w_2|^2)}\lesssim \|  {\mathbf f}\|^p_{W^{k, p}(\mathbb H)}.
    \end{split}
\end{equation}
Combining \eqref{dd}-\eqref{ee}, we have the desired inequality \eqref{tuu}. 

\eqref{tuf} follows from \eqref{tuu} and \eqref{t2}. 
To see \eqref{p2},  we shall verify that $\bar\partial^m_{w_2}\partial^l_{w_2}\tilde u(w_1, 0) =0$ for all $l, m\in \mathbb Z^+\cup\{0\}, l+m\le k-1$. Note that $\bar\partial^m_{w_2}\partial^l_{w_2}\tilde u(w_1, \cdot) \in C^\alpha(\triangle) $ for some $\alpha>0$ by \eqref{tuf}. If $m=0$, then   $\partial^l_{w_2}\tilde u(w_1, 0) =0 $ by its definition. If $m\ge 1$, since $\bar\partial_{w_2} \tilde u = \tilde f_2$ by \eqref{hot},  
  $$ \bar\partial^m_{w_2}\partial^l_{w_2}\tilde u(w_1, 0)  =  \bar\partial^{m-1}_{w_2}\partial^l_{w_2}\tilde f_2(w_1, 0) =0, $$
where we used  \eqref{p2f} in the last equality. Thus  \eqref{p2} is proved, and the proof of the lemma is complete. 

\end{proof}

In order to derive the refined weighted estimate of $\tilde u$ in Proposition \ref{soe}, we also need the following  modified identities/formulas  for  $W^{k,p}$ functions on $\triangle$ with vanishing $(k-1)$-th Taylor polynomials. 

\begin{lem}\label{el}  Let $h\in W^{k, p}(\triangle, |w|^2), k\in \mathbb Z^+, p>4 $ with $\mathcal P_k h =0$. Then for a.e.   $w\in \triangle$,  \\
i).    $$2\pi i w^{-k}   h(w ) = \int_{b\triangle}\frac{   h ( \zeta)} {\zeta^{k}(\zeta-w )}d\zeta - \int_{ \triangle}\frac{ \bar \partial h(\zeta)} {\zeta^{k}(\zeta-w )}d\bar\zeta\wedge d\zeta; $$   
 ii).  $$ Th(w) - \tilde{\mathcal P}_k (Th)(w) = w^kT\left(w^{-k}h\right)(w), $$
where $\tilde {\mathcal P}_k$ is the $(k-1)$-th  order holomorphic Taylor polynomial operator  at $0$.
 \end{lem}
\begin{proof} 
For part i), applying the Cauchy-Green formula to $w^{-k}   h $ on $\triangle \setminus  \overline{\triangle_\epsilon}$, we have for each fixed $w\ne 0$, 
\begin{equation}\label{cg}
    2\pi i w^{-k}   h(w ) = \int_{b\triangle}\frac{   h ( \zeta)} {\zeta^{k}(\zeta-w )}d\zeta  - \int_{b\triangle_\epsilon}\frac{   h ( \zeta)} {\zeta^{k}(\zeta-w )}d\zeta- \int_{ \triangle\setminus  \overline{\triangle_\epsilon(0)}}\frac{  \bar \partial h(\zeta)} {\zeta^{k}(\zeta-w )}d\bar\zeta\wedge d\zeta. 
\end{equation} 
We claim that $$ \lim_{\epsilon\rightarrow 0}  \int_{b\triangle_\epsilon}\frac{   h ( \zeta)} {\zeta^{k}(\zeta-w )}d\zeta =  0. $$
Indeed, let $g_w(\zeta): =(\zeta-w )^{-1} h(\zeta) $. Since $w\ne 0$,  $g_w \in W^{k,p}(\triangle_\epsilon, |\zeta|^2), p>4$ with $\epsilon$ sufficiently small  and $ \mathcal P_k g_w =0$. In particular, $g_w\in C^{k-1, \alpha}(\triangle_\epsilon)$ for some $\alpha>0$, with $ |g_w(\zeta)|\lesssim |\zeta|^{k-1+\alpha}$ near $0$. Thus
\begin{equation}\label{st2}
    \begin{split} \lim_{\epsilon\rightarrow 0}\left| \int_{b\triangle_\epsilon}\frac{   h ( \zeta)} {\zeta^{k}(\zeta-w )}d\zeta\right| \le     \lim_{\epsilon\rightarrow 0}   \epsilon^{-k} \int_{b\triangle_\epsilon} |g_w( \zeta)|   d\sigma_\zeta \lesssim   \lim_{\epsilon\rightarrow 0} \epsilon^{\alpha} =0. 
     \end{split}
 \end{equation}
The claim is proved. Part i) follows from the claim by letting $\epsilon \rightarrow 0$ in \eqref{cg}. 

For ii), let $\chi$ be a smooth function which is 1  near $0$, and vanishes  outside  $\triangle_{\frac{1}{2}} $. A direct computation gives that 
\begin{equation*}
    \begin{split}
-2\pi i    \partial     Th(0) = &\left.\partial  \int_\triangle \frac{\chi(\zeta) h(\zeta)}{\zeta-w}d\bar\zeta\wedge d\zeta \right|_{w=0}+\left.\partial \int_\triangle \frac{(1-\chi(\zeta)) h(\zeta)}{ \zeta-w }d\bar\zeta\wedge d\zeta\right|_{w=0} \\
=& \left. \int_{\mathbb C} \frac{\partial_w\left( \chi(\zeta+w) h(\zeta+w)\right)}{\zeta}d\bar\zeta\wedge d\zeta\right|_{w=0} +\int_\triangle \frac{(1-\chi(\zeta)) h(\zeta)}{\zeta^{2}}d\bar\zeta\wedge d\zeta\\
=&\int_{\mathbb C} \frac{\partial_\zeta\left( \chi(\zeta ) h(\zeta )\right)}{\zeta}d\bar\zeta\wedge d\zeta  +\int_\triangle \frac{(1-\chi(\zeta)) h(\zeta)}{\zeta^{2}}d\bar\zeta\wedge d\zeta\\
=&\int_{\mathbb C} \frac{  \chi(\zeta ) h(\zeta ) }{\zeta^2}d\bar\zeta\wedge d\zeta  +\int_\triangle \frac{(1-\chi(\zeta)) h(\zeta)}{\zeta^{2}}d\bar\zeta\wedge d\zeta = \int_\triangle \frac{  h(\zeta)}{\zeta^{2}}d\bar\zeta\wedge d\zeta.
    \end{split}
\end{equation*}
Here in the fourth line above we used Stokes' theorem and a similar argument as in \eqref{st2} (with $k=1$ there). Consequently with an induction, 
$$\tilde{\mathcal P}_k Th = -\sum_{l=0}^{k-1}\frac{w^l}{2\pi i} \int_\triangle \frac{h(\zeta)}{\zeta^{l+1}}d\bar\zeta\wedge d\zeta.   $$
Note that each term in the right hand side of the above is well defined due to Remark \ref{re2}.

Making use of the  following elementary identity for the Cauchy kernel: 
   \begin{equation*} 
     \frac{1}{ \zeta-w} -\sum_{l=0}^{k-1}\frac{w^{l}}{\zeta^{l+1}} = \frac{w^k}{\zeta^k(\zeta-w)},\ \ \text{for all}\ \ \zeta\ne  w\ \text{nor}\ \  0,
\end{equation*}
we immediately  get
\begin{equation*}
    \begin{split}
   Th(w) - \tilde{\mathcal P}_k Th(w) =  -    \frac{w^k}{2\pi i} \int_\triangle \frac{h(\zeta)}{\zeta^{k}(\zeta-w)}d\bar\zeta\wedge d\zeta = w^k T\left( w^{-k}h\right),\ \ \ \ w\in \triangle. 
    \end{split}
\end{equation*}

\end{proof}

\begin{lem}\label{2s}
  If $h\in W^{2, p}(\triangle, |w|^2), p>4$, then   $$  S\partial h = \partial S h  +S(\bar w^2\bar\partial h)\ \ \text{on} \ \ \triangle.  $$
\end{lem}

\begin{proof}

Note that $h\in W^{2, p}(\triangle, |w|^2)\subset C^{1,\alpha}(\triangle)$ for some $\alpha>0$. So both sides of the above equality are actually in the strong sense. The lemma  follows from  a  direct computation below.   For   $w\in \triangle$, 
\begin{equation*}
    \begin{split}
     S \partial h(w)  &=  \frac{1}{2\pi i } \int_{0}^{{2\pi}}  \frac{\partial_\zeta h(e^{i\theta}) ie^{i\theta}}{ e^{i\theta} - w }   d\theta =  \frac{1}{2\pi i  } \int_{0}^{{2\pi}}  \frac{\partial_\theta\left( h(e^{i\theta})\right) +i\bar\partial_\zeta h(e^{i\theta})e^{-i\theta}}{ e^{i\theta} - w }   d\theta\\
     &= - \frac{1}{2\pi i } \int_{0}^{{2\pi}} \partial_\theta \left(\frac{1}{ e^{i\theta} - w }\right)    h(e^{i\theta})d\theta    +\frac{1}{2\pi i   } \int_{0}^{{2\pi}}  \frac{ \bar\partial_\zeta h(e^{i\theta})e^{-2i\theta}}{ e^{i\theta} - w }  ie^{i\theta} d\theta\\
    & =   \frac{1}{2\pi i } \int_{0}^{{2\pi}} \partial_w \left(\frac{1}{ e^{i\theta} - w }\right)    h(e^{i\theta})ie^{ i\theta}d\theta + \frac{1}{2\pi i } \int_{b\triangle}  \frac{  \bar\partial_\zeta h(\zeta)\bar\zeta^2}{ \zeta - w }   d \zeta\\
   & =   \frac{1}{2\pi i } \int_{b\triangle} \partial_w \left(\frac{1}{ \zeta- w }\right)    h(\zeta)d\zeta +S\left(\bar w^2\bar\partial h\right) =   \partial  S   h(w)  +S\left(\bar w^2\bar\partial h\right)(w).
    \end{split}
\end{equation*}
\end{proof}


\medskip

\begin{proof}[Proof of Proposition \ref{soe}: ] In view of Lemma \ref{tu},  we only need to prove the estimate in the proposition when  $s\le k-1$. 
   First consider the case when  $0\le t\le k-1$. For  fixed $w_1\in \triangle$, 
     $h_{w_1} : = \partial_{ w_2}^{s} \tilde  u(w_1, \cdot)\in W^{k-s, p}(\triangle, |w_2|^2)$, $\mathcal P_{k-s}h_{w_1} =0$ by \eqref{p2}, and    $\bar\partial_{w_2}  h_{w_1} =   \partial_{ w_2}^{s} \tilde  f_2 $. We apply Lemma \ref{el}, part i) 
to  $h_{w_1} $ and obtain
$$2\pi i w_2^{-k+s} \partial_{ w_2}^{s} \tilde  u(w_1, w_2) = \int_{b\triangle}\frac{ \partial_{\zeta}^{s} \tilde  u (w_1, \zeta)} {\zeta^{k-s}(\zeta-w_2)}d\zeta - \int_{ \triangle}\frac{ \partial_{\zeta}^{s} \tilde  f_2(w_1, \zeta)} {\zeta^{k-s}(\zeta-w_2)}d\bar\zeta\wedge d\zeta.  $$
Consequently, 
\begin{equation*}
    \begin{split}
         w_2^{-k+s} \partial_{ w_1}^t \partial_{ w_2}^{s} \tilde  u (w_1, w_2) = &\frac{1}{2\pi i}\left(\partial_{w_1}^{t}\int_{b\triangle}\frac{   \partial_{ \zeta}^{s}\left(\bar\zeta^{k-s}\tilde  u (w_1, \zeta)\right)} {\zeta-w_2}d\zeta - \int_{ \triangle}\frac{ \zeta^{-k+s}  \partial_{ w_1}^t\partial_{ \zeta}^{s} \tilde  f_2(w_1, \zeta)} {\zeta-w_2}d\bar\zeta\wedge d\zeta\right)\\
        =&   \partial_{w_1}^{t}S_2\left(  \partial_{ w_2}^{s} \left(\bar w_2^{k-s}    \tilde  u\right)\right) +  T_2\left(w_2^{-k+s}  \partial_{ w_1}^t\partial_{ w_2}^{s}\tilde  f_2 \right) \\ 
        =&: B_1+B_2. 
    \end{split}
\end{equation*}

By  \eqref{T_j} and Proposition \ref{tf},
\begin{equation*}
    \begin{split}
     \left\|B_2\right\| _{L^p({\triangle\times \triangle}, |w_2|^2)}\lesssim
       &\left\|T_2\left(w_2^{-k+s}  \partial_{ w_1}^t\partial_{ w_2}^{s}\tilde  f_2 \right)\right\|_{L^p({\triangle\times \triangle}, |w_2|^2)}\lesssim  \left\| w_2^{-k+s}  \partial_{ w_1}^t\partial_{ w_2}^{s}\tilde  f_2  \right\|_{L^p({\triangle\times \triangle}, |w_2|^2)}\lesssim \left\| \mathbf f \right\|_{  W^{k,p}(\mathbb H)}.   
    \end{split}
\end{equation*} 
For $B_1$, if $s=0$, then $  B_1 =  S_2\left( \bar w_2^{k}  \partial_{ w_1}^{t }   \tilde  u\right)   $, where  
 $   \bar w_2^{k} \partial_{w_1}^{t }\tilde  u \in W^{1, p}({\triangle\times \triangle}, |w_2|^2)$ as $t\le k-1$. Then   \eqref{S_j} and Lemma \ref{tu} give 
\begin{equation*}
    \begin{split}
     \left\|B_1\right\|_{L^p({\triangle\times \triangle}, |w_2|^2)}\lesssim&   \left\|  S_2\left( \bar w_2^{k}  \partial_{ w_1}^{t }   \tilde  u\right)     \right\|_{L^{p}({\triangle\times \triangle}, |w_2|^2)}\lesssim    \left\|   \bar w_2^{k}  \partial_{ w_1}^{t }   \tilde  u      \right\|_{W^{1, p}({\triangle\times \triangle}, |w_2|^2)} \\
     \lesssim& \|\tilde u\|_{W^{k, p}({\triangle\times \triangle}, |w_2|^2)}\lesssim \left\| \mathbf f \right\|_{  W^{k,p}(\mathbb H)}.       \end{split}
\end{equation*}
For the case  $s \ge 1$, since $s\le k-1$,  $  \partial_{ w_2}^{s-1}\left(\bar w_2^{k-s}\tilde  u\right)(w_1, \cdot) \in W^{2, p}(\triangle, |w_2|^2)  $ for fixed $w_1\in \triangle$. Applying Lemma \ref{2s}  to $ \partial_{ w_2}^{s-1}\left(\bar w_2^{k-s}\tilde  u\right)(w_1, \cdot) $ and using the fact that  $\bar\partial_{w_2} \tilde u = \tilde f_2$, we further write  
\begin{equation*}
    \begin{split}
         B_1 =  & \partial_{w_1}^{t}\partial_{w_2}S_2\left( \partial_{ w_2}^{s-1}\left(\bar w_2^{k-s}\tilde  u\right)\right) +  \partial_{w_1}^{t} S_2\left( \bar w_2^2\partial_{ w_2}^{s-1}\left((k-s)\bar w_2^{k-s-1}\tilde  u + \bar w_2^{k-s}\tilde f_2\right) \right)\\
         =& \partial_{w_2}S_2\left(\partial_{w_1}^{t} \partial_{ w_2}^{s-1}\left(\bar w_2^{k-s}\tilde  u\right)\right) + (k-s)   S_2\left( \partial_{w_1}^{t}\partial_{ w_2}^{s-1}\left(\bar w_2^{k-s+1}\tilde  u\right) \right) +S_2\left( \partial_{w_1}^{t}\partial_{ w_2}^{s-1}\left( \bar w_2^{k-s+2}\tilde f_2\right) \right).  
    \end{split}
\end{equation*}
Note that $\partial_{w_1}^{t} \partial_{ w_2}^{s-1}\left(\bar w_2^{l}\tilde  u\right)\in W^{1, p}({\triangle\times \triangle}, |w_2|^2)$ for $ l = k-s, k-s+1, k-s+2$. By    \eqref{S11}, Proposition \ref{tf}  and \eqref{tuf},
\begin{equation*}
    \begin{split}
     \left\|B_1\right\|_{L^p({\triangle\times \triangle}, |w_2|^2)}\lesssim&  \left\|     \partial_{ w_1}^{t } \partial_{w_2}^{s -1}  \left(\bar w_2^{k-s}\tilde  u\right)   \right\|_{W^{1, p}({\triangle\times \triangle}, |w_2|^2)} +\left\|     \partial_{ w_1}^{t } \partial_{w_2}^{s -1}  \left(\bar w_2^{k-s+1}\tilde  u\right)   \right\|_{W^{1, p}({\triangle\times \triangle}, |w_2|^2)} \\
     &+\left\|     \partial_{ w_1}^{t } \partial_{w_2}^{s -1}  \left(\bar w_2^{k-s+2}\tilde  f_2\right)   \right\|_{W^{1, p}({\triangle\times \triangle}, |w_2|^2)}\\
     \lesssim& \|\tilde u\|_{W^{k, p}({\triangle\times \triangle}, |w_2|^2)}+  \left\|  \tilde{   f}_2 \right\|_{  W^{k,p}({\triangle\times \triangle}, |w_2|^2)}\lesssim \left\| \mathbf f \right\|_{  W^{k,p}(\mathbb H)}.       \end{split}
\end{equation*}

Finally, we treat the case when  $t=k$ (and so $s=0$).  According to the  definition of $\tilde u$, 
\begin{equation*}
    \begin{split}
        \tilde u = &T_1\tilde f_1 + S_1T_2 \tilde f_2 - T_1\tilde{\mathcal P}_{2, k}\tilde f_1 - S_1\tilde{\mathcal P}_{2, k} T_2 \tilde f_2\\
        =& T_1\tilde f_1 + S_1\left(T_2 - \tilde{\mathcal P}_{2, k} T_2 \right) \tilde f_2\\
        = & T_1\tilde f_1 + S_1\left(w_2^kT_2\left(w_2^{-k}  \tilde f_2\right)\right).
    \end{split}
\end{equation*}
Here  we   used the fact that $ \mathcal P_{2, k}\tilde f_1 =0$ by   \eqref{p2f}     in the second equality, and Lemma \ref{el}  part ii) in the third equality  for each fixed $w_1\in \triangle$. Consequently, 
\begin{equation*}
    \begin{split}
        w_2^{-k } \partial_{ w_1}^k   \tilde  u  = & \partial_{ w_1}^kT_1 \left( w_2^{-k }  \tilde f_1\right)  +     T_2  \left (\partial_{ w_1}^k S_1 \left(w_2^{-k} \tilde f_2\right)\right) =: C_1+C_2. 
    \end{split}
\end{equation*}
For $C_1$, by  \eqref{T11}   and Proposition \ref{tf} (with $s=0$ there),
\begin{equation*}
    \begin{split}
         \left\|C_1\right\|_{L^p({\triangle\times \triangle}, |w_2|^2)}\lesssim &  \sum_{j=0}^{k-1}\left\| w_2^{-k} \nabla_{w_1}^j\tilde f_2   \right\|_{L^{p}({\triangle\times \triangle}, |w_2|^2)}   \lesssim \|\mathbf f\|_{W^{k, p}(\mathbb H)}.   
    \end{split}
\end{equation*}
For $C_2$, by  \eqref{T11} (with $k=1$ there),  \eqref{S11} and Proposition \ref{tf}  (with $s=0$ there).
\begin{equation*}
    \begin{split}
         \left\|C_2\right\|_{L^p({\triangle\times \triangle}, |w_2|^2)}\lesssim &  \left\|\partial_{ w_1}^k S_1 \left(w_2^{-k} \tilde f_2\right)   \right\|_{L^p({\triangle\times \triangle}, |w_2|^2)}  \lesssim \sum_{j=0}^k \left\| w_2^{-k} \nabla^j_{w_1}\tilde f_2  \right\|_{L^{p}({\triangle\times \triangle}, |w_2|^2)}\lesssim \|\mathbf f\|_{W^{k, p}(\mathbb H)}.   
    \end{split}
\end{equation*}
The   proof of the proposition is thus complete.

\end{proof}

\subsection{Proof of the main theorem}
\begin{proof}[Proof of Theorem \ref{main}: ]
Let $\mathcal T_k \mathbf f: = \phi^*\tilde u + u_k$ on $\mathbb H$, where $\tilde u$ is defined in \eqref{tud}, and $u_k$ satisfies \eqref{pu}. Then $\bar\partial \mathcal T_k \mathbf f = \mathbf f$ on $\mathbb H$. To show the desired estimate for $\|\mathcal T_k \mathbf f\|_{ W^{k, p}(\mathbb H)}$, since the anti-holomorphic derivatives of $\mathcal T_k \mathbf f$  are shifted to that of ${\mathbf f}$, we only need to estimate $ \left\|\partial_{ z_1}^{l_1}\partial_{ z_2}^{l_2}   \left(\phi^*\tilde u\right)\right\|_{L^p(\mathbb H)}$, $l_1, l_2\in \mathbb Z^+\cup\{0\}, l_1+l_2\le k$. Note that 
$$ \partial_{ z_1}^{l_1}\partial_{ z_2}^{l_2} \left(\phi^*\tilde u\right)  = \sum_{s+t \le l_1+l_2, t\ge l_1} C_{l_1, l_2, t, s}z_1^{t-l_1}z_2^{-t-l_2+s}  \left( \partial_{ w_1}^{t}\partial_{ w_2}^{s}\tilde  u\right) \left(\frac{z_1}{z_2}, z_2\right)$$
for some constants $C_{l_1, l_2, t, s} $ dependent on $l_1, l_2, t, s$, and $|z_1|\le |z_2|$ on $\mathbb H$. Then by a change of variables, 
\begin{equation*}
    \begin{split}
    \left\|\partial_{ z_1}^{l_1}\partial_{ z_2}^{l_2}  \left(\phi^*\tilde u\right)\right\|_{L^p(\mathbb H)}\lesssim& \sum_{s+t\le l_1+l_2, t\ge l_1}\left\||w_2|^{-l_1-l_2+s}   \partial_{ w_1}^{t}\partial_{ w_2}^{s}\tilde  u  \left(w_1, w_2\right)  \right\|_{L^p({\triangle\times \triangle}, |w_2|^2)}\\
    \le&\sum_{s+t\le k} \left\||w_2|^{-k +s}  \partial_{ w_1}^{t}\partial_{ w_2}^{s}\tilde  u \left(w_1, w_2\right)  \right\|_{L^p({\triangle\times \triangle}, |w_2|^2)}.
    \end{split}
\end{equation*}
The rest of the proof follows from Proposition \ref{soe}. 

\end{proof} 
\medskip

The following Kerzman-type example demonstrates that the $\bar\partial$ problem on $\mathbb H$ with $W^{k, p}$ data in general does not expect solutions in $W^{k, p+\epsilon}$,  $\epsilon>0$, which verifies the optimality of Theorem \ref{main}.

\begin{example}\label{ex}
  For each $k\in \mathbb Z^+$ and $ 2<p<\infty$, let $\mathbf f= (z_2-1)^{k-\frac{2}{{p}}}d\bar z_1 $ on $\mathbb H$, $\frac{1}{2}\pi <\arg (z_2-1)<\frac{3}{2}\pi$. Then $\mathbf f\in W^{k, \tilde  p}(\mathbb H)$ for all $2<\tilde  p< p$ and   is  $\bar\partial$-closed on $\mathbb H$. However, there does not exist a solution $u\in W^{k,p}(\mathbb H)$ to $\bar\partial u =\mathbf f$ on $\mathbb H$.  
 \end{example}

 \begin{proof}
 Clearly $\mathbf f\in W^{k, \tilde  p}(\mathbb H) $ for all $2<\tilde  p< p$ and    is $\bar\partial$-closed on $\mathbb H$. 
Arguing by contradiction, suppose there exists some $u\in W^{k,p}(\mathbb H )$ satisfying $\bar\partial u =\mathbf f $ on $\mathbb H$. In particular, since $\triangle_{\frac{1}{2}}\times(\triangle \setminus \overline{\triangle_{\frac{1}{2}}}) \subset \mathbb H $, there exists some holomorphic function $h$ on $\triangle_{\frac{1}{2}}\times(\triangle \setminus \overline{\triangle_{\frac{1}{2}}})$ such that $  u  |_{ \triangle_{\frac{1}{2}}\times(\triangle \setminus \overline{\triangle_{\frac{1}{2}}})}= (z_2-1)^{k-\frac{2}{{p}}}\bar z_1+h \in W^{k,p}(\triangle_{\frac{1}{2}}\times(\triangle \setminus \overline{\triangle_{\frac{1}{2}}}))$. 

For  each fixed $(r, z_2) \in U: =  \left(0,\frac{1}{2}\right)\times \left( \triangle\setminus \overline{ \triangle_{\frac{1}{2}}}\right)\subset \mathbb R\times \mathbb C$, consider  
    $$v(r, z_2): =\int_{|z_1|= r} {\tilde u}(z_1, z_2) dz_1. $$
Then  with  a similar argument as in the proof of Example \ref{ex2}, one can see that $v\in W^{k,p}(U)$. 
Note that   $h(\cdot, z_2)$ is holomorphic on $\triangle_{\frac{1}{2}}$ for each fixed $z_2\in \triangle\setminus \overline{\triangle_\frac{1}{2}}$. Thus  for   fixed $(r, z_2)\in U$,  Cauchy's theorem gives
  \begin{equation*}
     v(r, z_2) =\int_{|z_1|=r} z_2(z_2-1)^{k-\frac{2}{{p}}}\bar z_1dz_1  = 2\pi r^2i z_2(z_2-1)^{k-\frac{2}{{p}}},
  \end{equation*}
  which does not belong to  $W^{k,p}(U)$. A contradiction!
  
 \end{proof}

\bibliographystyle{alphaspecial}

\begin{thebibliography}{99}

 \fontsize{9.5pt}{9pt} \selectfont

\setlength{\parskip}{1ex}
\setlength{\itemsep}{1ex}











\bibitem{BFLS}{\sc Burchard, A.; Flynn, J.; Lu, G.; Shaw, M.}: {\em Extendability and the $\bar\partial$ operator on the Hartogs triangle.} Math. Z. {\bf 301}(2022), no. 3, 2771--2792.


\bibitem{CC} {\sc  Chaumat, J.;  and  Chollet, A.-M.}:   {\em R\'egularit\'e h\"old\'erienne de l'op\'erateur $\bar\partial$ sur le triangle de Hartogs}, Ann. Inst. Fourier (Grenoble) {\bf 41}(1991), no. 4, 867--882.

\bibitem{CM} {\sc Chen, L.; McNeal, J. }:   {\em A solution operator for $\bar\partial$ on the Hartogs triangle and $L^p$ estimates.} Math. Ann. {\bf 376}(2020), no. 1-2, 407--430.
   



\bibitem{Ch} {\sc  Chua, S.}: {\em Extension theorems on weighted Sobolev spaces.} Indiana Univ. Math. J. \textbf{41}(1992), no. 4, 1027--1076.

\bibitem{CS}{\sc Chakrabarti, D.; Shaw, M.-C.}: {\em Sobolev regularity of the $\bar\partial$-equation on the Hartogs triangle}, Math. Ann. {\bf 356}(2013), no. 1, 241--258.






\bibitem{GR}{\sc Garc\'a-Cuerva, J.; Rubio de Francia, J.}: {\em Weighted norm inequalities and related topics.} North-Holland Mathematics Studies, 116. Notas de Matem\'atica [Mathematical Notes], 104. North-Holland Publishing Co., Amsterdam, 1985. x+604 pp.


\bibitem{GO} {Gehring, F. W.; Osgood, B. G.}: {\em Uniform domains and the quasihyperbolic metric.} J. Analyse Math. {\bf 36}(1979), 50-–74.



\bibitem{JY}{\sc Jin, M.; Yuan, Y.}:  {\em On the canonical solution of $\bar\partial$ on polydiscs}, C. R. Math. Acad. Sci. Paris. {\bf 358}(2020), no. 5, 523--528.


\bibitem{Jo}{\sc Jones, P.}: Quasiconformal mappings and extendability of functions spaces, Acta Math. {\bf 147}(1981), 71-88.

\bibitem{Ma}{\sc Maz'ya, V.}; {\em Sobolev spaces with applications to elliptic partial differential equations.} Second, revised and augmented edition. Grundlehren der mathematischen Wissenschaften [Fundamental Principles of Mathematical Sciences], 342. Springer, Heidelberg, 2011. xxviii+866 pp. 

\bibitem{MM} {\sc  Ma, L; Michel, J.}: {\em $C^{k+a}$-estimates for the $\bar\partial$-equation on the Hartogs triangle}, Math. Ann. {\bf 294}(1992), no. 4, 661--675.


\bibitem{NW}{\sc  Nijenhuis, A.; Woolf, W.}: {\em
Some integration problems in almost-complex and complex manifolds.}
Ann. of Math. (2) {\bf 77}(1963), 424--489.







\bibitem{PZ}{ \sc Pan, Y.; Zhang, Y.}: {\em H\"older estimates for the 
$\bar\partial$ problem for $(p.q)$ forms on product domains.} International Journal of Mathematics. \textbf{32}(2021), no. 3, 20 pp.

\bibitem{PZ2}{ \sc Pan, Y.; Zhang, Y.}: {\em Weighted Sobolev estimates of the truncated Beurling operator. } Preprint. 



\bibitem{Sh} {\sc Shaw, M. }: {\em    The Hartogs Triangle in Complex Analysis.} Geometry and topology of submanifolds and currents, 105–115, Contemp. Math., 646, Amer. Math. Soc., Providence, RI, 2015.


\bibitem{Stein}{\sc Stein, E.}: {\em Harmonic analysis: real-variable methods, orthogonality, and oscillatory integrals}.  Princeton University Press, Princeton, NJ, 1993. xiv+695 pp.


\bibitem{V} {\sc Vekua, I.}: {\em Generalized analytic functions},  vol.~29, Pergamon Press Oxford, 1962.


\bibitem{YZ} {\sc Yuan, Y.; Zhang, Y.}:  {\em Weighted Sobolev estimates of $\bar\partial$ on  domains covered by polydiscs.} Preprint. 
 





\bibitem{Zhang2} {\sc Zhang, Y.}: {\em Optimal $L^p$  regularity for $\bar\partial$ on the Hartogs triangle.} Preprint.  ArXiv:2207.04944.

\end{thebibliography}

\fontsize{11}{11}\selectfont
\vspace{0.7cm}

\noindent pan@pfw.edu,

\vspace{0.2 cm}

\noindent Department of Mathematical Sciences, Purdue University Fort Wayne, Fort Wayne, IN 46805-1499, USA.\\

\vspace{0.5cm}

\noindent zhangyu@pfw.edu,

\vspace{0.2 cm}

\noindent Department of Mathematical Sciences, Purdue University Fort Wayne, Fort Wayne, IN 46805-1499, USA.\\
\end{document}